\documentclass[12pt]{article}
\usepackage{amsthm,amsmath,amssymb,amscd,graphics,enumerate,latexsym,verbatim}
\usepackage[all]{xy}
\usepackage{ulem}	

\makeatletter
\renewcommand*\@fnsymbol[1]{\the#1}
\makeatother

\theoremstyle{plain}
\newtheorem{theorem}{Theorem}[section]
\newtheorem{corollary}[theorem]{Corollary}
\newtheorem{lemma}[theorem]{Lemma}
\newtheorem{proposition}[theorem]{Proposition}

\theoremstyle{definition}
\newtheorem{definition}[theorem]{Definition}

\theoremstyle{remark}
\newtheorem{remark}[theorem]{Remark}

\newcommand{\bbf}{\mathbb{F}}

\def\S{{\text{S}}}

\numberwithin{equation}{section}

\DeclareMathOperator{\GL}{GL}
\DeclareMathOperator{\Quot}{Quot}
\DeclareMathOperator{\spec}{spec}
\DeclareMathOperator{\Spec}{Spec}
\DeclareMathOperator{\res}{res}
\DeclareMathOperator{\Hom}{Hom}

\DeclareMathOperator{\Pic}{Pic}
\DeclareMathOperator{\Bun}{Bun}
\DeclareMathOperator{\Proj}{Bun}
\DeclareMathOperator{\Coh}{Coh}
\DeclareMathOperator{\QCoh}{qCoh}
\DeclareMathOperator{\colim}{colim}
\DeclareMathOperator{\coker}{coker}
\DeclareMathOperator{\im}{im}

\DeclareMathOperator{\eq}{eq}
\DeclareMathOperator{\coeq}{coeq}
\DeclareMathOperator{\tensor}{\otimes}
\def\Mod{\mathcal{M}od}

\def\0{{\bf 0}}
\def\A{{\mathbb A}}
\def\F{{\mathbb F}}
\def\G{{\mathbb G}}
\def\P{{\mathbb P}}
\def\Z{{\mathbb Z}}
\def\m{{\mathfrak m}}
\def\p{{\mathfrak p}}
\def\q{{\mathfrak q}}
\def\cB{{\mathcal B}}
\def\cC{{\mathcal C}}
\def\cD{{\mathcal D}}
\def\cF{{\mathcal F}}
\def\cG{{\mathcal G}}
\def\cK{{\mathcal K}}
\def\cL{{\mathcal L}}
\def\cM{{\mathcal M}}
\def\cN{{\mathcal N}}
\def\cO{{\mathcal O}}
\def\cP{{\mathcal P}}

\def\cS{{\mathcal S}}
\def\cT{{\mathcal T}}
\def\cU{{\mathcal U}}
\def\cV{{\mathcal V}}
\def\cX{{\mathcal X}}
\def\cY{{\mathcal Y}}

\def\Fun{{\F_1}}
\def\Gm{\mathbb G_m}
\def\Mo{{\mathcal M_0}}

\def\int{\textup{int}}

\def\id{\textup{id}}
\def\msch{\widetilde}
\newdir{ >}{{}*!/-5pt/@^{(}}\newcommand{\arincl}[1]{\ar@{ >->}@<-0,0ex>#1} 

\title{Sheaves and $K$-theory for $\mathbb{F}_1$-schemes}
\author{Chenghao Chu\footnote{Department of Mathematics, Johns Hopkins University, 3400 N. Charles Street, Baltimore MD 21218, USA. cchu@math.jhu.edu, +1-410-516-5128.},
 Oliver Lorscheid\footnote{Department of Mathematics, University of Wuppertal, Gau\ss str. 20, 42097 Wuppertal, Germany. lorscheid@math.uni-wuppertal.de, +49-202-439-2523 \emph{Corresponding author.}},
 Rekha Santhanam\footnote{{Department of Mathematics and Statistics, IIT Kanpur, Kanpur-208016, India. reksan@iitk.ac.in, +91512-259-6950.}}}

\begin{document}

\maketitle

\begin{abstract}
 This paper is devoted to the open problem in $\Fun$-geometry of developing $K$-theory for $\Fun$-schemes. We provide all necessary facts from the theory of monoid actions on pointed sets and we introduce sheaves for $\Mo$-schemes and $\Fun$-schemes in the sense of Connes and Consani. A wide range of results hopefully lies the background for further developments of the algebraic geometry over $\Fun$. Special attention is paid to two aspects particular to $\Fun$-geometry, namely, normal morphisms and locally projective sheaves, which occur when we adopt Quillen's Q-construction to a definition of $G$-theory and $K$-theory for $\Fun$-schemes. A comparison with Waldhausen's $S_{\bullet}$-construction yields the ring structure of $K$-theory. In particular, we generalize Deitmar's $K$-theory of monoids and show that $K_*(\Spec\Fun)$ realizes the stable homotopy of the spheres as a ring spectrum. 
\end{abstract}

\pagebreak

\tableofcontents

\section{Introduction}\label{Recollections and Notations}

In the late 1980's and early 1990's, several ideas shaped a philosophy of what a geometry over $\Fun$, the \emph{field with one element}, should be and which statements should be satisfied (cf.\ \cite{KapranovUN, Manin}). While the main drive was and is the hope to transfer Weil's proof of the Riemann hypothesis from positive characteristics to $\mathbb Q$ by interpreting $\Spec \Z$ as a \emph{scheme over $\Fun$}, the original idea of Jacques Tits (\cite{Tits}) also played an important role in the development of $\Fun$-geometry. Rephrased in nowadays language, Tits proposed that reductive groups should be defined over $\Fun$ and that the $\Fun$-rational points should have the natural structure of the Weyl group of the reductive group.

Tits' idea led further to the expectation that there is a $K$-theory for $\Fun$-schemes with the property that $K_*(\Spec\Fun)$ realizes $\pi_*^s(\mathbb{S})$, the stable homotopy groups of the sphere spectrum (cf.\ \cite{Manin, Soule2004}); namely, one would like to be able to formulate an equation of the form
\begin{equation} \label{sphere}\tag{$\ast$}
 K_*(\Spec\bbf_1) \quad = \quad \pi_*({B\GL(\infty,\Fun)}^+) \quad  = \quad \pi_*({B\Sigma_\infty}^+) \quad \simeq  \quad \pi_*^s(\mathbb{S}) \end{equation}
where the first equality is the definition of $K$-theory via Quillen's +-construc\-tion, naively applied to the elusive field $\Fun$. The equality in the middle is derived from Tits' idea, since the Weyl group of $\GL(n)$ is the symmetric group $\Sigma_n$, and therefore
$$ \GL(\infty,\Fun) \quad = \quad \bigcup_{n\geq 1} \GL(n,\Fun) \quad = \quad \bigcup_{n\geq 1} \Sigma_n \quad = \quad \Sigma_\infty. $$
The last isomorphism in equation \eqref{sphere} is the Barratt-Priddy-Quillen theorem (\cite{Barratt, Priddy}).

So far, this philosophy is partially realized by Deitmar's definitions of $K$-theory for semigroups with a unit (see \cite{Deitmar06}). Deitmar adapted Quillen's +-construction and Q-construction to semigroups, which correspond to affine $\Fun$-schemes as commutative rings correspond to affine schemes. Both theories give the expected outcome $K_*(\Fun)\simeq\pi_*^s(\mathbb{S})$ if one defines $\Fun$ as the trivial monoid $\{1\}$.

From a different point of view, without using the notion of $\Fun$-schemes, H\"uttemann e.a.\ considered the algebraic $K$-theory of projective spaces over monoids which are called nonlinear projective spaces in \cite{H02, H-W01}. They employed Waldhausen's construction to the category of sheaves of these spaces.

The purpose of this paper is to generalize the work on algebraic $K$-theory of affine $\Fun$-schemes, i.e.\ semi-groups as in \cite{Deitmar06}, and projective spaces as in \cite{H02} to general $\Fun$-schemes. In the last few years, around a dozen different definitions of an $\Fun$-scheme were given by generalizing scheme theory from different viewpoints (cf.\ the overview paper \cite{LL09b}). One of this notion was introduced by Connes and Consani in \cite{CC09}. In the present paper, we will follow their approach since it is the only one in which Tits' idea was realized so far (see \cite{L09}). Consequently, we will always refer to Connes and Consani's definition when we consider an $\Fun$-scheme. 

Our basic definition of $K$-theory adopts Quillen's Q-construction to the context of $\Fun$-schemes. A final comparison of the Q-construction with Waldhausen's $S_\bullet$-construction shows that the $K$-theory of an $\cM_0$-scheme is indeed a symmetric $E_\infty$-ring spectrum.

We briefly review the Q-construction applied to $\Fun$ since it gives a good idea of the more complicated construction of $K$-theory for a general $\Fun$-scheme. We consider the category $\Mod_f\ \Fun$ of finite pointed sets together with base point preserving maps. An \emph{admissible monomorphism} is an injective morphism and an \emph{admissible epimorphism} is a surjective morphism whose fibres, except for the fibre of the base point, contain precisely one element. We apply the Q-construction: $Q\Mod_f\ \Fun$ is the category whose objects are finite pointed sets and whose morphisms are isomorphism classes of diagrams
$$\xymatrix{N&\ M\ar@{>->}[l]\ar@{->>}[r]&P}$$
where the first arrow is an admissible monomorphism and the second arrow is an admissible epimorphism. The $i$-th $K$-group $K_i(\Fun)$ is defined as $\pi_{i+1}(BQ\Mod\ \Fun)$, the $(i+1)$-th homotopy group of the classifying space of $Q\Mod_f\ \Fun$, which is computed in the proof of Theorem \ref{KTheoryOfSpecF1[H]}.

The paper is organized as follows. In Section \ref{TheSectionOnMonoids}, we set up a theory of monoids, i.e.\ semigroups with a unit and with a zero, or, an absorbing element. In particular, we introduce and investigate $A$-sets which will play the role of modules over the monoid $A$. Some parts of this section are, in spirit, covered by Deitmar's papers on $\Fun$ (\cite{Deitmar05, Deitmar06, Deitmar07}) and by the theory of $A$-acts (see \cite{MUA})---however,  we require monoids to have a zero and acts to be sets with a base point.  This is not exactly the case in the previously mentioned works; therefore we take the opportunity to give a self-contained treatment of commutative algebra for monoids.

In Section \ref{geometry_of_monoids}, we recall the notion of an $\Mo$-scheme as introduced by Connes and Consani in \cite{CC09}, which is the analogue of a scheme when rings are replaced by monoids. Then we introduce sheaves for $\Mo$-schemes. The sheaf theory for $\Mo$-schemes behaves in many aspects like usual sheaf theory---the notion of $\cO_X$-modules makes sense, and coherent sheaves can be defined by a local property---but they show certain different behaviors: not all epimorphisms are normal and projective $\cO_X$-modules (in the categorical sense) are, in general, not locally free, but only locally projective. Only little parts of this section are covered in literature yet, basically only the definition of an $\Mo$-scheme and a few remarks on $\cO_X$-modules and coherent sheaves by Deitmar in his theory of $\Fun$-schemes (see \cite{Deitmar05}).

In Section \ref{SheavesOfModulesOnF1Schemes}, we review Connes and Consani's definition of an $\Fun$-scheme. We introduce sheaves for $\Fun$-schemes, based on the previous sections, and provide the necessary theory for the definition of $K$-theory. In particular, we define normal morphisms and admissible exact sequences.

In Section \ref{subsection:definition_of_k-theory}, we show that coherent sheaves over $\Fun$-schemes together with admissible exact sequences form a quasi exact category, which makes it possible to define $G$-theory for $\Fun$-schemes. From this, we deduce that the notion of admissible exact sequences also leads to a definition of $K$-theory. This definition generalizes Deitmar's definition via the Q-construction, and in particular realizes the stable homotopy groups of the sphere  as the $K$-groups of $\Spec\Fun$. Usage of Waldhausen's $S_\bullet$-construction allows us to talk about the $K$-theory spectrum of an $\Fun$-scheme, and it turns out that this is indeed an $E_\infty$- symmetric ring spectrum in the case of $\cM_0$-schemes. In particular, this ring structure is compatible with the ring structure of the  sphere spectrum.

{\bf Acknowledgements:} The authors like to thank Professor Consani for many helpful discussions on this work, and the referee for many inspiring suggestions. The main part of the paper was written while the first and third authors were affiliated to the Johns Hopkins University. The second author likes to thank the mathematical institute of the University of Wuppertal for supporting the project \emph{Cohomology over $\Fun$}. The third author would like to thank Bjorn Dundas for his suggestions.

\section{Commutative algebra for monoids} \label{TheSectionOnMonoids}

The study of schemes depends largely on the study of commutative algebras over rings. Similarly, the study of $\Fun$-schemes depends largely on the commutative algebra over monoids.  In the first subsection, we recall basic definitions and facts on monoids (cf.\ \cite{Kato94}, \cite{Deitmar05}, and \cite{CC09}) and complete the picture by some new insights. In the second subsection, we define and investigate $A$-sets (cf.\ the notion of $A$-acts in \cite{K72}), which play the role of modules over a monoid in view towards $\Fun$-geometry.

\subsection{ Monoids}

We will introduce the category $\Mo$ of monoids and provide general facts about limits and colimits. Then we study localizations of monoids at multiplicative subsets, the base extension to $\Z$ and properties of finitely generated monoids. We round of this section on monoids by a list of examples.

\subsubsection{Definition and general properties}

A \emph{monoid} is a (multiplicatively written) commutative  semi-group $A$ with a \emph{zero} (also called an \emph{absorbing element}) and a \emph{one}, i.e.\ elements $0$ and $1$ that satisfy $0\cdot a=0$ and $1\cdot a=a$, respectively, for all $a\in A$. A \emph{morphism of monoids} is a multiplicative map that preserves $0$ and $1$. Following \cite{CC09}, we denote the category of monoids by $\Mo$.

The category $\Mo$ will be interpreted as the category of $\Fun$-algebras. It has an initial object, namely the monoid $\{0,1\}$ with distinctive $0$ and $1$, which we will denote from now on by $\Fun$. The terminal object of $\Mo$ is the \emph{zero monoid $\{0\}$} with one element $0=1$.

Recall that a directed diagram is a commutative diagram where for every pair of objects $A_i$ and $A_j$, there are an object $A_k$ and morphisms $A_i\to A_k$ and $A_j\to A_k$.

\begin{proposition} \label{monoid-limits}
 The category $\Mo$ contains small limits, finite coproducts and colimits of directed diagrams.
\end{proposition}

\begin{proof}
 To prove that $\Mo$ contains small limits, it suffices to prove that $\Mo$ contains small products and equalizers (cf.\ \cite[Thm.\ 2.8.1]{Borceux}). The product of a family of monoids $\{A_i\}_{i\in I}$ is given by the Cartesian product $\prod A_i$ over $I$ together with componentwise multiplication and componentwise projections to the $A_i$. Its zero is the element whose components are all zero, and its one is the element whose components are all one. The universal property of a product is verified immediately.

 The equalizer of two monoid morphisms $f,g:A\to B$ is the submonoid $\eq(f,g)=\{a\in A\mid f(a)=g(a)\}$ of $A$. Since $f(0)=0=g(0)$ and $f(1)=1=g(1)$, the equalizer contains $0$ and $1$, and since $f(ab)=f(a)f(b)=g(a)g(b)=g(ab)$ for all $a,b\in \eq(f,g)$, the set $\eq(f,g)$ is multiplicatively closed and thus a monoid. The submonoid $\eq(f,g)$ obviously satisfies the universal property of an equalizer of $f$ and $g$ since equalizers are monomorphisms and monomorphisms in $\Mo$ are injective maps.

 The coproduct of a finite family $\{A_i\}_{i\in I}$ is given by the smash product $\bigwedge A_i$ over $I$, which is the quotient of the Cartesian product $\prod A_i$ by the equivalence relation that identifies every element with a component that is $0$ with the zero element $(0)_{i\in I}$ in $\prod A_i$. Multiplication in $\bigwedge A_i$ is defined componentwise, the zero is the class of $(0)_{i\in I}$ and the one is the element $(1)_{i\in I}$ whose components are all one. There are canonical inclusions $A_j\to \prod A_i$ that send $a\in A_j$ to the element $(a_i)$ with $a_i=a$ for $i=j$ and $a_i=1$ otherwise. To verify the universal property, let $\{f_i:A_i\to B\}_{i\in I}$ be a family of monoid morphisms. Then the morphism $f:\bigwedge A_i \to B$ sending $(a_i)$ to $f((a_i))=\prod f_i(a_i)$ satisfies the universal property of a coproduct.

 Let $\cD=\{A_i\}_{i\in I}$ be a commutative diagram of monoids and morphisms indexed by a directed set $I$, i.e.\ for every $i,j\in I$, there is a $k\in I$ and (unique) morphisms $f_i:A_i\to A_k$ and $f_j:A_j\to A_k$ in $\cD$. For $i\in I$, define $J(i)$ as the cofinal directed subset $\{k\in I\mid \exists f:A_i\to A_k\text{ in }\cD\}$ of $I$, and let $\cD(i)$ be the full subdiagram of $\cD$ that contains precisely $\{A_i\}_{i\in J(i)}$. Then the colimit of $\cD$ can be represented by
 $$ \colim \cD \quad = \quad \coprod_{i\in I}\ \biggl\{ (a_j)\in \prod_{j\in J(i)}A_j\biggl| \forall f: A_j\to A_k\text{ in } \cD(i),\, a_k=f(a_j)\biggr\}\; \biggl/ \; \sim $$
 where two elements $(a_j)_{j\in J(i_1)}$ and  $(b_j)_{j\in J(i_2)}$ are equivalent if $a_j=b_j$ for all $j\in J(i_1)\cap J(i_2)$. The canonical morphisms $\iota_i:A_i\to\colim \cD$ map $a_i \in A_i$ to $(f(a_i)\mid f:A_i\to A_k\text{ in }\cD(i))$. Given a family of monoid morphisms $g_i:A_i\to B$ that commute with all morphisms in $\cD$, the map $\displaystyle g:\colim \cD \to B$ that sends an element $(a_j)_{j\in J(i)}$ to $g_i(a_i)$ is the unique morphism that satisfies the universal property of the colimit of $\cD$.
\end{proof}

\begin{remark}
 In section \ref{tensor_products}, we will see that $\Mo$ contains all pushouts of diagrams of the form $B\leftarrow A\to C$ in $\Mo$. Note that this implies the existence of finite coproducts since the coproduct of $B$ and $C$ is the pushout of $B\leftarrow\Fun\to C$.
\end{remark}

A subset $I\subset A$ is called an \emph{ideal} if $I$ is not empty and if $IA\subset I$. In particular, an ideal must contain $0$. A subset $S\subset A$ is called a \emph{multiplicative set} if $1\in S$ and if $ab\in S$ for all $a,b\in S$. An ideal $\p\subset A$ is called \emph{prime ideal} if its complement $A-\p$ is a multiplicative set or, equivalently, if for all $a,b\in A$ with $ab\in\p$ either $a\in\p$ or $b\in\p$ and if $1\notin\p$. If $f:A\to B$ is a monoid morphism and $\p$ is a prime ideal of $B$, then $f^{-1}(\p)$ is a prime ideal of $A$. Every monoid has a unique maximal prime ideal, namely, $A-A^\times$ where $A^\times$ is the \emph{group of units}.

An element $a\in A$ is called \emph{nilpotent} if there is a natural number $n$ such that $a^n=0$. The set of all nilpotent elements of $A$ is an ideal, called the \emph{nilradical of $A$}. The following is proven exactly as in the case of commutative rings; for example, see \cite[Proposition 1.8]{AM}.

\begin{lemma}
 The nilradical of $A$ is the intersection of all prime ideals of $A$, which is the same as the intersection of all minimal prime ideals.
\end{lemma}

\subsubsection{Localization}
\label{localization_of_monoids}

Let $A$ be a monoid and $S\subset A$ a multiplicative set. We define the \emph{localization of $A$ at $S$} as $ S^{-1}A = (A\times S) / \sim$ where the equivalence relation $\sim$ is defined by the rule that $(a,s)\sim(a',s')$ if and only if there is a $t\in S$ such that $tas'=ta's$. We write $\frac as$ for the pair $(a,s)$. The association $a\mapsto \frac a1$ defines a monoid morphism $A\to S^{-1}A$. Note that like in the case of rings, the map $A\to S^{-1}A$ is an epimorphism.

Let $A^\int$ be the set of all elements $a$ of $A$ such that multiplication by $a$ defines an injective map. We say that a monoid $A$ is \emph{integral} if $A=A^\int\cup\{0\}$. If $A$ is integral, then we define the \emph{quotient monoid $\Quot A$} as $(A-\{0\})^{-1} A$.  Note that the canonical morphism $A\to\Quot A$ is injective and that $\Quot A-\{0\}$ is a group.

For every $f\in A$, the set $S_f=\{f^i\}_{i\geq0}$ is a multiplicative set. We write $A_f$ for $S_f^{-1}A$. If $\p\subset A$ is a prime ideal, then we denote  the localization of $A$ at $S=A-\p$ by $A_\p$.  Note that $S^{-1}A=(A^\times S)^{-1}A$, that $S^{-1}A$ is the zero monoid $\{0=1\}$ if and only if $0\in S$, and that $S^{-1}A=A$ if and only if $S\subset A^\times$.

For a multiplicative subset $S$ of $A$, we denote by $U_S$ the subset of all prime ideals of $A$ that do not intersect $S$. Note that $U_S$ is empty if and only if $0\in S$ and that $U_S=U_{A^\times S}$.

\begin{lemma} \label{prime_ideal_in_U_S}
 There is a one to one order preserving bijection between $U_S$ and the prime ideals of $S^{-1}A$.
\end{lemma}

\begin{proof}
 Let $g:A\to S^{-1}A$ be the canonical map. The bijection is given by mapping a prime ideal $\p$ of $A$ that does not intersect $S$ to the ideal $g(\p)S^{-1}A$, which is easily verified to be prime. The inverse of this mapping sends a prime ideal $\q$ of $S^{-1}A$ to $g^{-1}(\q)$. It is clear that these inverse maps are order preserving.
\end{proof}

\begin{corollary} \label{maximal_ideal_in_U_S}
 If  $U_S$ is not empty, then there is a prime ideal $\m$ in $U_S$ that contains all ideals of $A$ that do not intersect $S$. In particular, $\m$ contains all other prime ideals of $U_S$, and $S^{-1}A=A_\m$.
\end{corollary}

\begin{proof}
 It is clear from the previous lemma that the inverse image of the maximal ideal $S^{-1}A-(S^{-1}A)^\times$ of $S^{-1}A$ under the canonical map $A\to S^{-1}A$ is the prime ideal $\m$ with the required property. Since the image of an element $a\in A-\m$ in $A_\m$ is a unit, $S^{-1}A=A_\m$.
\end{proof}

Let $f\in A$. For $S=\{f^i\}_{i\geq0}$, we put $D_f=U_S$, which is the set of all prime ideals of $A$ that do not contain $f$. The following is analogous to ring theory.

\begin{lemma} \label{D_f-topology}
 Let $f,g\in A$. Then $D_f\cap D_g=D_{fg}$. The set $U_0$ is empty and $U_1$ is the set of all prime ideals of $A$.\qed
\end{lemma}

\subsubsection{Base extension to $\Z$}

Let $A$ be a monoid. We define the \emph{base extension of $A$ to $\Z$} as the ring $A_\Z=\Z[A]/(1\cdot 0_A)$, which is the semi-group ring of $A$ modulo the ideal generated by the zero $0_A$ of $A$.

\begin{lemma}\label{localization_and_base_extension}
 Let $A$ be monoid and $S$ a multiplicative subset. The canonical ring homomorphism $S^{-1}A_\Z \to (S^{-1}A)_\Z$, defined by linear extension of the map that sends $\frac as$ to $\frac as$ for $a\in A$ and $s\in S$, is an isomorphism.
\end{lemma}

\begin{proof}
 This follows from the fact that we can rewrite an element $\sum_{i=1}^n m_i\frac{a_i}{s_i} \in (S^{-1}A)_\Z$ as
$$ \sum_{i=1}^n m_i\frac{a_i}{s_i} \ = \ \sum_{i=1}^n \frac{m_ia_i\prod_{j\neq i}s_j}{\prod_{j=1}^ns_j} \ = \frac{\sum_{i=1}^n m_i a_i'}s $$
where $a_i'=a_i\prod_{j\neq i}s_j\in A$ and $s=\prod_{j=1}^ns_j\in S$. This is an element of $S^{-1}A_\Z$. It is clear that this defines an inverse map to the canonical ring homomorphism $S^{-1}A_\Z \to (S^{-1}A)_\Z$.
\end{proof}

The following is an easy fact.

\begin{lemma}
 A monoid morphism $A\to B$ is injective (surjective) if and only if $A_\Z\to B_\Z$ is injective (surjective).\qed
\end{lemma}

\subsubsection{Finitely generated monoids}

A \emph{set of generators of a monoid $A$} is a subset $\Gamma$ of $A$ such that every element $f\in A-\{0,1\}$ can be written as a product of elements of $\Gamma$. A \emph{set of generators of a multiplicative subset $S$ of $A$} is a subset $\Gamma$ of $S$ such that every element $f\in S-\{1\}$ can be written as a product of elements in $\Gamma$. A \emph{set of generators of an ideal $I$ of $A$} is a subset $\Gamma$ of $I$ such that every element of $I$ can be written as a product of an element of $\Gamma$ by an element of $A$. A monoid (resp.\ multiplicative subset resp.\ ideal) is \emph{finitely generated} if one can choose a finite set of generators.

\begin{lemma}\label{finitely_many_prime_ideals}
 Let $A$ be generated by a subset $\Gamma$. Then every prime ideal of $A$ is generated by a subset of $\Gamma$. In particular, if $\Gamma$ is finite, then every prime ideal of $A$ is finitely generated and $A$ has only finitely many prime ideals.
\end{lemma}

\begin{proof}
 Let $\p$ be a prime ideal and $f\in\p$. Then $f$ can be written as a product of elements in $\Gamma$. Since $\p$ is a prime ideal it must contain one of the factors of $f$. This shows that $\p$ is generated by $\Gamma\cap\p$. The second claim of the lemma follows from the first claim.
\end{proof}

Note that a finitely generated monoid has typically infinitely many ideals. For instance, the monoid $\Fun[T]=\{T^i\}_{i\geq0}\cup\{0\}$ contains for every $k\geq 0$ the ideal $\langle T^i\rangle=\{T^i\}_{i\geq k}\cup\{0\}$.

\begin{proposition}\label{U_S=D_f}
 Let $A$ be a finitely generated monoid. Then there is for every multiplicative subset $S$ of $A$ an element $f\in A$ such that $U_S=D_f$.
\end{proposition}

\begin{proof}
 Let $\Gamma$ be a finite set that generates $A$ and $S$ a multiplicative set. By Corollary \ref{maximal_ideal_in_U_S}, the set $U_S$ of prime ideals not intersecting $S$ contains a prime ideal $\m$ that contains all other prime ideals in $U_S$. By Lemma \ref{finitely_many_prime_ideals}, $\m$ is generated by $\Gamma\cap\m$. Let $f$ be the product of all elements in $\Gamma$ that are not contained in $\m$. Since every prime ideal $\p$ is generated by $\Gamma\cap\p$ by the previous lemma, $\p\in D_f$ if and only if $\p\cap\Gamma\ \subset\ \m\cap\Gamma$ or, simply, $\p\subset\m$. This is, as mentioned, the condition for $\p$ to be in $U_S$. Thus we have proven that $U_S=D_f$.
\end{proof}

The following is analogue to \cite[Lemma 2]{Deitmar06} and the same proof applies.

\begin{lemma}\label{monoid_of_finite_type_and_base_extension}
 A monoid $A$ is finitely generated if and only if $A_\Z$ is finitely generated as a ring.\qed
\end{lemma}

\subsubsection{Examples}

Earlier, we already introduced the monoid $\Fun=\{0,1\}$ and the zero monoid $\{0\}$ with one element $0=1$. We call the monoid $\Fun[T]=\{T^i\}_{i\geq0}\cup\{0\}$ the \emph{polynomial ring over $\Fun$} because its base extension $\Fun[T]_\Z$ is isomorphic to $\Z[T]$ and because it plays in $\Mo$ the role of a free algebra on one generator over $\Fun$. The quotient monoid $\Quot\Fun[T]$ is the monoid $\{T^i\}_{i\in\Z}\cup\{0\}$, which we will denote by $\Fun[T,T^{-1}]$. Its base extension to $\Z$ is the ring $\Z[T,T^{-1}]$.

For every commutative semi-group $A$ with $1$, we can define a monoid $\Fun[A]=A\cup \{0\}$ that extends the multiplication of $A$ by $a\cdot 0=0$ for all $a\in \Fun[A]$. In particular, every abelian group $G$ gives rise to a monoid $\Fun[G]$, which defines an embedding of the category of abelian groups into $\Mo$.

If $A$ is a monoid, then $A^\times$ and $A^\int$ are commutative semi-groups with $1$, and thus define submonoids $\Fun[A^\times]$ respective $\Fun[A^\int]$ of $A$.

If $I$ is an ideal of a monoid $A$, then we define the quotient monoid $A/I$ as the set $(A-I)\cup\{0\}$ together with the multiplication $a\cdot b=ab$ if $a,b\in A-I$ such that $ab\notin I$ and $a\cdot b=0$ otherwise. In particular, $A-A^\times$ and $A-A^\int$ are ideals of $A$, and we have $\Fun[A^\times]\simeq A/(A-A^\times)$ and $\Fun[A^\int]\simeq A/(A-A^\int)$.

Finally, every ring $R$ defines a monoid by forgetting its addition.


\subsection{$A$-sets}

For semigroups $A$ that do not necessarily contain a zero or an one, the notion of an \emph{$A$-act} is defined and well-studied in literature. An $A$-act is a set together with an action by $A$. The $A$-acts play the role of modules over the semigroups $A$---not to be confused with the notion of an \emph{$A$-module} in case that $A$ is a group, which is a module over the group ring $\Z[A]$. However, from the viewpoint of $\Fun$-geometry, we require monoids $A$ to have a  zero and an one, and consequently a theory of modules over $A$ leads to a different kind of objects. To avoid confusion with the existing notions, we call the objects of our studies \emph{$A$-sets} and investigate them in detail in the following section.

\subsubsection{Definition and general properties}

For a pointed set $M$ with base point $\ast$, the set $\Hom(M,M)$ of base point preserving self-maps $M\to M$ together with composition is a (generally non-commutative) semi-group. It has a $0$, namely, the map sending all elements to the base point, and it has a $1$, namely, the identity. Let $A$ be a monoid. An \emph{$A$-set} is a pointed set $M$ together with a multiplicative map $A\to\Hom(M,M)$ that preserves $0$ and $1$.

In other words, an $A$-set is a pointed set $M$ together with an \emph{$A$-action}, i.e.\ a map $\theta: A\times M\to M$ satisfying the following properties for all $a,b\in A$ and $m\in M$ (we will write $a.m$ for $\theta(a,m)$):
$$ (ab).m=a.(b.m), \qquad a.\ast=\ast, \qquad 0.m=\ast \qquad\text{and}\qquad 1.m=m. $$

A \emph{morphism of $A$-sets} $M$ and $N$ is an \emph{$A$-equivariant map}, i.e.\ a map $f:M\to N$ such that $f(a.m)=a.f(m)$ for all $a\in A$ and $m\in M$. In particular, a morphism of $A$-sets sends the base point to the base point. We denote the category of $A$-sets by $A-\Mod$.

The trivial $A$-set $0=\{\ast\}$ is a zero object of $A-\Mod$, i.e.\ both an initial and a terminal object. Consequently, there is for all $A$-sets $M$ and $N$ a unique morphism $\ast$ in the homomorphism set $\Hom(M,N)$ that factors through $0$, namely, the morphism that sends all elements of $M$ to the base point of $N$. The $A$-action $a.f:m\mapsto f(a.m)$ makes $\Hom(M,N)$ into an $A$-set. The $A$-set $\Hom(M,N)$ is functorial in both $M$ and $N$, thus $\Hom(-,-)$ is a bifunctor from $A-\Mod$ into itself.

The image of a morphism $f:M\to N$ of $A$-sets is the subset $\im f=\{n\in N\mid \exists m\in M, f(m)=n\}$ together with the induced $A$-action. Since $\ast=f(\ast)$ and $a.n=a.f(m)=f(a.m)$ for $n=f(m)$, this is indeed an $A$-subset of $N$.

\begin{lemma}\label{mono-epi}
 A morphism of $A$-sets is a monomorphism (epimorphism, isomorphism) if and only if it is an injective (surjective, bijective) map.
\end{lemma}

\begin{proof}
 Since a morphism of $A$-sets is defined as a map with additional properties, it is characterized by the image of each element. Thus it is clear that an injective morphism is a monomorphism and that a surjective map is an epimorphism. It is also clear that the inverse map of a bijective morphism between $A$-modules is $A$-equivariant. We proceed with the reverse implications.

 For the following arguments, we consider $A$ as an $A$-set by $a.b=ab$ for $a,b\in A$. Let $A^\times$ be the set of all invertible elements of $A$. We consider $\Fun[A^\times]=A^\times\cup\{\ast\}$ as an $A$-set by defining $a.b=ab$ if $a,b\in A^\times$ and $a.b=\ast$ for all other combinations of $a\in A$ and $b\in \Fun[A^\times]$ (see section \ref{A-set_examples} for more details).

 Let $f:M\to N$ be a monomorphism and let $m,m'\in M$ be elements that are mapped to the same element $f(m)=f(m')$. Consider the morphism $g,h: A\to M$ of $A$-sets that are defined by $g(a)=a.m$ and $h(a)=a.m'$. Then $f\circ g(a)=f(a.m)=a.f(m)=a.f(m')=f(a.m')=f\circ h(a)$. Since $f$ is a monomorphism, $g=h$ and $m=m'$. Thus $f$ is injective.

 Let $f:M\to N$ be an epimorphism. Assume that there is an element $n\in N-\im f$. Then $a.n\in N-\im f$ for every $a\in A^\times$, because if $a.n=f(m)$, then $n=a^{-1}a.n=a^{-1}.f(m)=f(a^{-1}.m)$ would be in the image of $f$. Define $g:N\to \Fun[A^\times]$ by $g(a.n)=a.n$ and $g(n')=\ast$ if $n'$ is not of the form $a.n$ for some $a\in A^\times$, and let $h:N\to \Fun[A^\times]$ be the trivial morphism sending every $n\in N$ to $\ast$. Then $g\circ f(m)=\ast=h\circ f(m)$ for all $m\in M$, but $g\neq h$, which is a contradiction. Thus $f$ must be surjective.

 An isomorphism is both a monomorphism and an epimorphism. By the preceding, an isomorphism must be injective and surjective, and consequently bijective.
\end{proof}

\begin{proposition}\label{ASetsAreCompleteAndCocomplete}
 The category $A-\Mod$ contains small limits and small colimits.
\end{proposition}

\begin{proof}
 To show that $A-\Mod$ contains small (co)limits it is enough to show that it contains (co)products over arbitrary index sets and (co)equalizers, cf.\ \cite[Thm.\ 2.8.1]{Borceux}.

 The product of a family $\{M_i\}_{i\in I}$ of $A$-sets is the Cartesian product $\prod_{i\in I} M_i$ together with the diagonal action of $A$ and the projections $p_j: \prod M_i\to M_j$ surjective the $j$-th component, which are $A$-equivariant maps. The universal property of the product is easily verified.

 The coproduct of a family $\{M_i\}_{i\in I}$ of $A$-sets is the wedge product $\bigvee_{i\in I} M_i$, which is the quotient of the disjoint union of all $M_i$ by the equivalence relation that identifies the base points. The $A$-action on $\bigvee M_i$ is defined via the canonical inclusions $\iota_j: M_j\to \bigvee M_i$. Namely, $a.\iota_j(m)=\iota_j(a.m)$ for $a\in A$ and $m\in M_j$. It is easily verified that $\bigvee M_i$ satisfies the universal property of the coproduct.

 We proceed with the equalizer. Given two morphisms $f,g:M\to N$ of $A$-sets, the equalizer of $f$ and $g$ is $\eq(f,g)=\{m\in M\mid f(m)=g(m)\}$ together with the inclusion as a subset of $A$ and the restricted $A$-action. The subset $\eq(f,g)$ of $M$ is indeed an $A$-set since $f(\ast)=g(\ast)$ and $f(a.m)=a.f(m)=a.g(m)=g(a.m)$ if $f(m)=g(m)$. Since equalizers are monomorphisms and monomorphisms in $A-\Mod$ are injective by Lemma \ref{mono-epi}, it is clear that $\eq(f,g)$ satisfies the universal property of an equalizer.

 The coequalizer of $f$ and $g$ is the quotient $\coeq(f,g)=N/\sim$ by the equivalence relation generated by $n\sim n'$ if there is an $m\in M$ such that $n=f(m)$ and $n'=g(m)$, together with the quotient map $N\to \coeq(f,g)$ and the induced $A$-action. The $A$-action on $\coeq(f,g)$ is well-defined since $f(\ast)=g(\ast)$ and since for $n=f(m)$ and $n'=g(m)$, we have $a.n=a.f(m)=f(a.m)\sim g(a.m)=a.g(m)=a.n'$. Since coequalizers are epimorphisms and epimorphisms in $A-\Mod$ are surjective by Lemma \ref{mono-epi}, it is clear that $\coeq(f,g)$ satisfies the universal property of a coequalizer.
\end{proof}

The (co)kernel of a morphism $f:M\to N$ is defined as the (co)equalizer of the diagram
$$ \xymatrix{ M \ar@< 2pt>[r]^f \ar@<-2pt>[r]_{*} &  N}.  $$
This means that the kernel $\ker f$ of $f$ is the subset $f^{-1}(*)$ of $M$. The cokernel $\coker f$ is the quotient of $N$ by the equivalence relation defined as $n\sim n'$ if and only if $n=n'$ or $n,n'\in\im f$. We denote this quotient by $N/\im f$. This means that the quotient $N/I$ for any $A$-subset $I\subset N$ exists as it is the cokernel of the inclusion map.

A diagram $M_1\stackrel{f}{\to} M_2 \stackrel{g}{\to} M_3$ is said to be \emph{exact at $M_2$} if $\ker(g)=\im(f)$. A \emph{short exact sequence of $A$-sets} is a sequence of the form
$$ 0 \longrightarrow M_1 \longrightarrow M_2 \longrightarrow M_3 \longrightarrow 0 $$
that is exact at $M_1$, $M_2$ and $M_3$.

\subsubsection{Normal morphisms}

Recall that in a category with a zero object, a monomorphism is called \emph{normal} if it is a kernel and an epimorphism is called \emph{normal} if it is a cokernel.

\begin{lemma}\label{normal_mono_and_epi}
 All monomorphisms in $A-\Mod$ are normal, and an epimorphism in $A-\Mod$ is normal if and only if all its fibres contain at most one element except for the fibre of the base point.
\end{lemma}

\begin{proof}
 Let $f:M\to N$ be a monomorphism. Then $\im f$ is an $A$-subset of $N$ and we can consider the quotient map $g:N\to N/\im f$. Then $f=\ker g$ and thus $f$ is normal.

 Let $f:M\to N$ be a normal epimorphism, i.e.\ there is a morphism $g:P\to M$ such that $f=\coker g$. This means that $N\simeq M/\im g$ and $f$ is the quotient map $M \to M/\im g$, which is of the form as described in the lemma. If the fibres of $f$ contain at most one element, expect for the fibre of the base point, then $f$ is the cokernel of its own kernel $f^{-1}(\ast)\hookrightarrow M$ and thus a normal epimorphism.
\end{proof}

\begin{proposition}\label{normal_lemma}
 Let $f:M\to N$ be a morphism of $A$-sets. The following conditions are equivalent.
 \begin{enumerate}
  \item\label{normal1} There is a (normal) monomorphism $g:M\to P$ and a normal epimorphism $h:P\to N$ such that $f=h\circ g$.
  \item\label{normal2} There is a normal epimorphism $h:M\to P$ and a (normal) monomorphism $g:P\to N$ such that $f=g\circ h$.
  \item\label{normal3} Each fibre of $f$ contains at most one element except for the fibre $f^{-1}(\ast)$ of the base point $\ast\in N$.
  \item\label{normal4} The canonical morphism $M/\ker f\to\im f$ is an isomorphism.
 \end{enumerate}
 If $f$ satisfies these equivalent conditions, we say that $f$ is \emph{normal}. The composition of normal morphisms is normal.
\end{proposition}

See Lemma \ref{extension_of_normal_morphism_to_Z} for another characterization of normal morphisms.

\begin{proof}
 In this proof, we will make frequent use of the characterization of (normal) monomorphisms, (normal) epimorphisms and isomorphisms as described in Lemmas \ref{mono-epi} and \ref{normal_mono_and_epi} without further reference. We begin with the equivalence of \eqref{normal3} and \eqref{normal4}. The canonical morphism $M/\ker f\to \im f$ is an isomorphism if and only if it is bijective. It is always surjective, and one sees at once that it is injective if and only if $f$ satisfies \eqref{normal3}.

 The following are some preparative observations. If $f$ is a monomorphism, then it always satisfies the different conditions of the proposition. If $f$ is an epimorphism with \eqref{normal3}, then obviously \eqref{normal1} and \eqref{normal2} hold for $f$. Finally note that morphisms with property \eqref{normal3} are closed under composition. In particular, the last statement of the proposition will follow from this once the equivalences are proven.

 Assume \eqref{normal1}, i.e.\ $f=h\circ g$ for a monomorphism $g$ and an epimorphism $h$. Then both $g$ and $h$ satisfy \eqref{normal1} for trivial reasons, and they satisfy \eqref{normal3} by the preceding considerations. As $f=g\circ h$, it does so as well. Similarly, \eqref{normal2} implies \eqref{normal3}.

 Assume \eqref{normal3}. Define $g:M\to N\vee\ker f$ by $g(m)=m\in\ker f$ if $m\in\ker f$ and $g(m)=f(n)\in N$ if $m\notin\ker f$. This is an injective $A$-equivariant map, and thus a (normal) monomorphism. Define $h$ as the canonical map $h:N\vee\ker f\to (N\vee\ker f)/\ker f=N$. This is a normal epimorphism since $f$ satisfies \eqref{normal3}. Thus $f=h\circ g$ satisfies \eqref{normal1}.

 Similarly, $f$ has a factorization into a normal epimorphism $h:M\to M/\ker f$ followed by the obvious morphism $g:M/\ker f\to N$. Since $f$ satisfies \eqref{normal3}, $g$ is injective and thus a (normal) monomorphism. Hence $f=g\circ h$ satisfies \eqref{normal2}.
\end{proof}

\begin{remark}\label{remark_abelian_category}
 As we have seen, the category $A-\Mod$ satisfies many properties of an abelian category: it has finite limits and colimits and thus in particular products and coproducts, kernels and cokernels, pullbacks and pushouts. Monomorphisms are normal, and a morphism is an isomorphism if and only if it is both a monomorphism and an epimorphism.

 However, the following facts show that $A-\Mod$ is not an abelian category: in general, epimorphisms are \emph{not} normal, the canonical morphism $M/\ker f\to\im f$ is \emph{not} an isomorphism, the canonical morphism $M\vee N\to M\times N$ is \emph{not} an isomorphism and the morphism set $\Hom(M,N)$ does \emph{not} have an intrinsic structure of an abelian group (where $M$ and $N$ are any $A$-sets and $f:M\to N$ is any morphism).

 These properties and problems will be inherited by the categories of (quasi-)coherent sheaves on $\Mo$-schemes resp.\ $\Fun$-schemes, as we will see later. They allow us to carry over methods from algebraic geometry to a far extend, but we have to treat certain points with care. In particular, we will meet the class of normal morphisms again, when we define $K$-theory in  section \ref{subsection:definition_of_k-theory}.
\end{remark}

\begin{definition} We call a short exact sequence of $A$-sets $$ 0 \longrightarrow M_1 \longrightarrow M_2 \longrightarrow M_3 \longrightarrow 0 $$ \emph{admissible} if and only if all morphisms in the sequence are normal.
\end{definition}

\subsubsection{Tensor products}
\label{tensor_products}

Let $M$ and $N$ be $A$-sets. We will define the \emph{tensor product $M\tensor_A N$} as a quotient of the coproduct $\bigvee_{M\times N} A$. Note that $\bigvee_{M\times N} A$ is equal to $(M\times N\times A)/\sim_\vee$ as a pointed set where $(m,n,a)\sim_\vee (m',n',a')$ if and only if $a=\ast =a'$, and $A$ acts only on the last factor.

Let $\sim$ be the equivalence relation on $\bigvee_{M\times N} A$ that is generated by relations of the form $(b.m,n,a)\sim(m,b.n,a)\sim(m,n,ba)$ where $a,b\in A$, $m\in M$ and $n\in N$. This equivalence relation is compatible with $\sim_\vee$ and with the $A$-action on $\bigvee_{M\times N}A$, and we can define the \emph{tensor product $M\otimes_AN$} as the quotient $\bigvee_{M\times N}A/\sim$. We write $m\otimes n$ for the element $(m,n,1)\in M\otimes_A N$. Since $(m,n,a)\sim(a.m,n,1)$, we can write every element of $M\otimes_AN$ in this form. The $A$-action on $M\otimes_AN$ looks like $a.(m\otimes n)=(a.m)\otimes n$, which is the same as $m\otimes(a.n)$.

An alternative description of the tensor product is given by the following. The map $p: M\times N\to M\tensor_A N$ that maps $(m,n)$ to $m\otimes n$ is surjective and \emph{$A$-biequivariant}, i.e.\ $f(a.m,n)=a.f(m,n)=f(m,a.n)$ for all $a\in A$, $m\in M$ and $n\in N$. Let $\sim_\times$ be the equivalence relation on $M\times N$ that is generated by $(a.m,n)\sim_\times(m,a.n)$ for $a\in A$, $m\in M$ and $n\in N$. Then the above map $M\times N\to M\tensor_A N$ induces a bijection $M\times N/\sim_\times\stackrel\sim\longrightarrow M\tensor_A N$ of pointed sets.

From this description one verifies easily the universal property of the tensor product as formulated as below.

\begin{lemma}
 For every $A$-biequivariant map $f:M\times N\to P$, there is a unique $A$-equivariant map $f':M\otimes_A N\to P$ such that the diagram
 $$ \xymatrix{M\times N\ar[d]_p\ar[rr]^{f} && P\\M\otimes_AN\ar[urr]_{f'}} $$
 commutes. Given an $A$-set $M$, the functor $M\otimes_A (-)$ is left-adjoint to $\Hom(M,-)$. \qed
\end{lemma}

Let $M$ and $N$ be $A$-sets. We denote by $M\wedge N=M \times N / (M \vee N)$ the \emph{smash product} of pointed sets. This is an $A$-set via the action defined by $a.(m,n)=(a.m,a.n)$ if both $a.m\neq\ast$ and $a.n\neq\ast$, and $a.(m,n)=\ast$ otherwise (where $a\in A$, $m\in M-\{\ast\}$ and $n\in N-\{\ast\}$). In the special case $A=\Fun$, we have $M\otimes_{\Fun}N\simeq M\wedge N$ as $A$-sets. For general $A$, however, there exists only an $A$-biequivariant map $M\wedge N\to M\otimes_AN$.

Let $f:A\to B$ be a morphism of monoids, $M$ an $A$-set and $N$ a $B$-set. Then a map $g:M\to N$ of pointed sets is said to be \emph{compatible with $f$} if the diagram
$$ \xymatrix{A\times M  \ar[rr]^{\theta_M}\ar[d]_{(f,g)}&&M\ar[d]^g\\ B\times N\ar[rr]^{\theta_N}&&N} $$
commutes. We can consider $B$ as an $A$-set by defining $a.b=f(a)b$ for $a\in A$ and $b\in B$. In this case, the $A$-set $M\otimes_AB$ inherits the structure of a $B$-set by defining $b.(m,c)=(m,bc)$ for all $b,c\in B$ and $m\in M$. This extends naturally to a functor
$$ -\otimes_AB:\ A-\Mod \quad \longrightarrow \quad B-\Mod, $$
which we call the \emph{base extension functor from $A$ to $B$}. This functor has a right adjoint, namely, every $B$-set $N$ can be considered as an $A$-set by letting $A$ act on $N$ via $f$. With this at hand, we see that a map $g:M\to N$ is compatible with $f:A\to B$ if and only if it is $A$-equivariant.

Let $A\to B$ and $A\to C$ be monoid morphisms. Then both $B$ and $C$ are $A$-sets. We can define a monoid structure on the $A$-set $B\otimes_AC$ by $(b\otimes c)\cdot (b'\otimes c')=(bb')\otimes(cc')$. The zero of $B\otimes_AC$ is $0\otimes 0$ and its one is $1\otimes 1$. Together with the canonical morphism $B\to B\otimes_AC$ sending $b$ to $b\otimes 1$ and $C\to B\otimes_AC$ sending $c$ to $1\otimes c$, the monoid $B\otimes_AC$ is the pushout of the diagram $B\leftarrow A\to C$ in the category $\Mo$. In particular, if $A=\Fun$, then $B\otimes_AC$ equals the coproduct $B\wedge C$.

\subsubsection{Base extensions to $\Z$}

If $M$ is an $A$-set, we denote by $M_\Z$ the free abelian group on the generators $M-\{*\}$. It has a natural $A_\Z$-module structure by linear extension of the $A$-action on $M$. This extends to a functor
$$-\otimes_AA_\Z: \ A-\mathcal{M}od \quad \longrightarrow \quad A_\Z-\mathcal{M}od, $$
which we call the \emph{base extension functor from $A$ to $A_\Z$}. More generally, if $A$ is a monoid, $B$ is a ring and $f:A\to B$ is a multiplicative map, then there exists a unique extension of $f$ to a ring homomorphism $f_\Z:A_\Z\to B$. Let $M$ be an $A$-set, then we define $M\otimes_AB$ to be the $B$-module $M_\Z\otimes_{A_\Z}B$. This defines the base extension functor $-\otimes_AB: A-\Mod\to B-\Mod$. In the special case $A=\Fun$ and $B=A_\Z=\Z$, we obtain the \emph{base extension functor from $\Fun$ to $\Z$}.

We collect some basic properties about base extensions that are easy to prove.
\begin{lemma}\label{easy_base_extension_properties}
Let $A$ be a monoid and let $f:M\to N$ be a morphism of $A$-sets.
\begin{enumerate}
 \item\label{base1} We have $(\coker f)_\Z\simeq\coker f_\Z$ and $(\im f)_\Z\simeq\im f_\Z$.
 \item\label{base2} We have $(M\otimes_AN)_\Z\simeq M_\Z\otimes_{A_\Z}N_\Z$ and $(M\vee N)_\Z\simeq M_\Z\oplus N_\Z$.
 \item\label{base3} If $A\to B$ is a morphism of monoids, then $(M\otimes_AB)_\Z\simeq M_\Z\otimes_{A_\Z}B_\Z$ as $B_\Z$-modules. If $N$ is a $B$-set and $M\to N$ is a morphism of pointed sets that is compatible with $A\to B$, then $M_\Z\to N_\Z$ is compatible with the ring homomorphism $A_\Z\to B_\Z$.
\item\label{base4} A morphism $f:M\to N$ of $A$-sets is injective (surjective) if and only if $f_\Z: M_\Z\to N_\Z$ is injective (surjective).\qed
\end{enumerate}
\end{lemma}

\begin{lemma} \label{extension_of_normal_morphism_to_Z}
 Let $f:M\to N$ be a morphism of $A$-sets. The canonical inclusion $\tau: (\ker f)_\Z\to\ker f_\Z$ is an isomorphism if and only if $f$ is normal.
\end{lemma}

\begin{proof}
 We have to prove that $\tau$ is surjective if and only if $f$ is normal. The surjectivity of $\tau$ is equivalent to the fact that the kernel of $f_\Z:M_\Z\to N_\Z$ is generated by elements $1\cdot m$, where $m\in\ker f$. This, in turn, is the case if and only if the set $M-\ker f$, which contains the basis elements of $M_\Z$ that are not contained in the kernel of $f_\Z$, is mapped injectively to the set $N-\{\ast\}$ of basis elements of $N_\Z$. This is equivalent saying that  $f$ is normal.
\end{proof}

\begin{remark}
 The circumstance that in general kernels do not commute with base extension to $\Z$ makes it necessary to perform certain constructions with care. For instance, the base extension of a short exact sequence
 $$ 0 \longrightarrow M_1 \longrightarrow M_2 \longrightarrow M_3 \longrightarrow 0 $$
 of $A$-sets to $\Z$ is always exact at $M_1\otimes_\Fun\Z$ and $M_3\otimes_\Fun\Z$, but it is exact at $M_2\otimes_\Fun\Z$ if and only if the epimorphism $M_2\to M_3$ is normal. Thus, in contrast to the base extension of rings, $(-)\otimes_\Fun\Z$ is right exact only for normal morphisms. The base extension is neither left exact, which means that the extension $\Z$ over $\Fun$ fails to be flat. This contrasts the intuition that $\Fun$ should behave like a field, but restricting to the class of normal morphisms fixes this defect.
\end{remark}

\subsubsection{Localization}

Let $S\subset A$ be a multiplicative set and $M$ be an $A$-set. We define the \emph{localization of $M$ at $S$} as the quotient $S^{-1}M=S\times M/\sim$ where the equivalence relation $\sim$ is defined by $(s,m)\sim(s',m')$ if and only if there is a $t\in S$ such that $tsm'=ts'm$. We write $\frac ms$ for elements $(s,m)$ of $S^{-1}M$. There is a canonical map $M\to S^{-1}M$ of pointed sets that sends $m$ to $\frac m1$. The set $S^{-1}M$ has the base point $\frac \ast 1$ and is an $S^{-1}A$-set by defining $(\frac as).(\frac mt)=\frac{a.m}{st}$. This extends naturally to a functor $S^{-1}: A-\Mod \to S^{-1}A-\Mod $.

If $S=\{f^n\}_{n\geq0}$ for some $f\in A$, then we define $M_f=S^{-1}M$. If $S=A-\p$ for a prime ideal $\p$ of $A$, then we define $M_\p=S^{-1}M$. The following statement is analogous to the case of modules over a ring.

\begin{lemma} \label{localization_of_A_and_M}
 Let $S\subset A$ be a multiplicative subset and $M$ be an $A$-set. Then $S^{-1}M\cong S^{-1}A\otimes_A M$ as $A$-sets.
\end{lemma}

\begin{proof}
 We verify that the maps
 $$ \begin{array}{cccc} \varphi: & S^{-1}M & \longrightarrow & S^{-1}A\otimes_A M \\ & \frac ms & \longmapsto & \frac 1s \otimes m \end{array} $$
 and
 $$ \begin{array}{cccc} \psi: & S^{-1}A\otimes_A M & \longrightarrow & S^{-1}M \\ & \frac as\otimes m& \longmapsto     & \frac {a.m}s \end{array} $$
 are well-defined. If $\frac ms=\frac{m'}{s'}$ in $S^{-1}M$, then there is a $t\in S$ such that $ts.m'=ts'.m$. Thus
 $$ \frac 1s\otimes m=\frac{ts'}{ts's}\otimes m=\frac{1}{tss'}\otimes ts'.m = \frac{1}{tss'}\otimes ts.m' = \frac{ts}{tss'}\otimes m' = \frac 1{s'}\otimes m' $$
 what shows that $\varphi$ is well-defined. To show that $\psi$ is well-defined, we first note that both $b.\frac as\otimes m$ and $\frac as \otimes b.m$ are mapped to $\frac {ab.m}s$ for $a,b\in A$, $s\in S$ and $m\in M$ which shows that $\psi$ is well-defined on the equivalence relation in the definition of the tensor product. To show that $\psi$ is also well-defined on the equivalence relation of the localization $S^{-1}A$, let $\frac as\otimes m= \frac{a'}{s'}\otimes m$ in $S^{-1}A\otimes_A M$, i.e.\ there is a $t\in S$ such that $ts'a=tsa'$. This implies
 $$ \frac{a'.m}{s'} = \frac{ta'.m}{ts'} = \frac{ta.m}{ts} = \frac{a.m}{s} $$
 which shows that $\psi$ is well-defined. It is now obvious that $\varphi$ and $\psi$ are mutually inverse morphisms of $A$-sets.
\end{proof}

\begin{lemma}\label{localization_and_base_extension_of_A-sets}
 Let $S\subset A$ be a multiplicative subset and $M$ be an $A$-set.
 \begin{enumerate}
  \item\label{locbase1} Let $f:A\to B$ be a morphism of monoids and $T=f(S)$, which is a multiplicative subset of $B$. Then $T^{-1}(M\otimes_AB)\simeq S^{-1}M\otimes_{S^{-1}A}T^{-1}B$.
  \item\label{locbase2}  There is an isomorphism $S^{-1}(M_\Z)\simeq (S^{-1}M)_\Z$.
 \end{enumerate}
\end{lemma}

\begin{proof}
 The proof of \eqref{locbase1} is analogous to the case of rings. The proof of \eqref{locbase2} is the same as the proof of Lemma \ref{localization_and_base_extension}.
\end{proof}

\begin{proposition}\label{prop_localizations_and_limits}
 Localizations of $A$-sets commute with finite limits and small colimits.
\end{proposition}

\begin{proof}
 Since finite colimits are equalizers of finite products and small colimits are coequalizers of small coproducts, it suffices to show that localizations commute with finite products, equalizers, small coproducts and coequalizers. Fix a multiplicative subset $S$ of $A$.

 Let $\{M_i\}$ be a finite collection of $A$-sets. Define $\Phi:S^{-1}\prod M_i\to \prod S^{-1}M_i$ by $\Phi(\frac{(m_i)}{s})=\bigr(\frac{m_i}s\bigl)$. It is easily verified that $\Phi$ is a morphism of $S^{-1}A$-sets and that $\Phi$ is injective. Surjectivity follows from the equation $\big(\frac{m_i}{s_i}\bigl)  = \bigl(\frac{s_i'm_i}{s}\bigr)$ in $\prod S^{-1}M_i$ where $s_i'=\prod_{j\neq i}s_j$ and $s=\prod s_i$ is the product over all $s_i$.

 Let $f,g:M\to N$ be two morphisms of $A$-sets. Then
 \begin{multline*}
   S^{-1}\eq(f,g) \ = \ \biggl\{\frac ms\in\S^{-1}M \biggl| f(m)=g(m)\biggr\} \\ = \ \biggl\{\frac ms\in\S^{-1}M \biggl| S^{-1}f(\frac ms)=S^{-1}g(\frac ms)\biggl\} \ = \  \eq(S^{-1}f,S^{-1}g)
 \end{multline*}

 Let $\{M_i\}$ be a family of $A$-set indexed by an arbitrary set. Then it is obvious that $S^{-1}\bigvee M_i\simeq \bigvee S^{-1}M_i$.

 Let $f,g:M\to N$ be two morphisms of $A$-sets. Then $S^{-1}\coeq(f,g)$ is by definition the quotient of $S\times N$ by the equivalence relation $\sim$ generated by $(s,n)\sim (s',n')$ if there is a $t\in S$ such that $ts.n'=ts'.n$ and $(s,n)\sim (s,n')$ if there is an $m\in M$ such that $n=f(m)$ and $n'=g(m)$. This is the same as $\coeq(S^{-1}f,S^{-1}g)$.
\end{proof}

\begin{lemma}\label{normal_and_localization}
 Let $S\subset A$ be a multiplicative subset. If $f:M\to N$ is a normal morphism of $A$-sets, then $S^{-1}f:S^{-1}M\to S^{-1}N$ is a normal morphism of $S^{-1}A$-sets.
\end{lemma}

\begin{proof}
 Let $\frac ms, \frac{m'}{s'}\in S^{-1}M$ be elements such that $S^{-1}f(\frac ms)=S^{-1}f(\frac{m'}{s'})$. Then, by definition of $S^{-1}f$, $\frac{f(m)}s=\frac{f(m')}{s'}$, which means that there is a $t\in S$ such that $f(ts.m')=ts.f(m')=ts'.f(m)=f(ts'.m)$. Since $f$ is normal, either $ts.m'=ts'.m$, which means that $\frac ms=\frac{m'}{s'}$, or $f(ts.m')=f(ts'.m)=\ast$. In the latter case, we can multiply the equation by $\frac1{tss'}$ and see that already $S^{-1}f(\frac ms)=S^{-1}f(\frac{m'}{s'})=\ast$. Thus $S^{-1}f$ is normal.
\end{proof}

\subsubsection{Projective $A$-sets}
It is well known that projective $A$-acts are disjoint union of $A$-acts of the form $eA$ where $e^2=e$, see \cite{K72}. We prove that the corresponding statement is also true for $A$-sets.

Let $S$ be a subset of an $A$-set $M$. The $A$-set $M$ is said to be \emph{free on $S$} if it satisfies the following universal property: for every $A$-set $N$ and every map $f:S\to N$ there is an $A$-equivariant map $F:M\to N$ such that $F(s)=f(s)$ for all $s\in S$. An $A$-set $M$ is said to be \emph{free} if there is a subset $S$ of $M$ such that $M$ is free on $S$.

One verifies immediately that for $S=\{s_i\}$, the $A$-set $\bigvee As_i$ is free on $S$. The universal property implies that $\bigvee As_i$ is the unique free $A$-set on $S$ up to unique isomorphism. The \emph{rank of a free $A$-set $M$} is the cardinality of $S$. In particular, the trivial $A$-set $\{\ast\}$ is a free $A$-set of rank $0$ on $S=\emptyset$.

By Lemma \ref{easy_base_extension_properties} \eqref{base2}, the base extension of a free $A$-set is a free $A_\Z$-module.

Recall that an object $P$ of a category is \emph{projective} if every morphism $P\to N$ factors through every epimorphism $M\to N$. Since $A-\Mod$ has the notion of exact sequences, an $A$-set $P$ is projective if and only if $\Hom(P,-)$ is exact. The universal property of a free $A$-set implies that every free $A$-set is projective. Another characterization of projective $A$-sets is the following.

\begin{lemma} \label{lemma:ProjectiveASetsSplitsIntoFree}
An $A$-set $P$ is projective if and only if there is a splitting epimorphism from some free $A$-set to $P$. An $A$-set $P=\bigvee_{i\in I}P_i$ is projective if and only if each $P_i$ is projective.
\end{lemma}
\begin{proof}
 Let $P$ be a projective $A$-set and $S=\{s_i\}\subset P$ be a set of generators. Then the canonical $A$-equivariant map $g:\bigvee As_i\to P$ is surjective and thus an epimorphism. Since $P$ is projective, there is a section $f:P\to \bigvee As_i$ of $g$.

 Conversely, if $g:\bigvee As_i\to P$ is an epimorphism with section $f:P\to \bigvee As_i$, we show that $P$ is projective. Given an epimorphism $j:M\to N$ and a morphism $k: P\to N$, we obtain the morphism $k\circ g:\bigvee As_i \to N.$ Since $\bigvee As_i$ is projective, $k\circ g$ can be lifted to $h: \bigvee As_i \to M$ such that $k\circ g = j\circ h$. Since $k = k\circ g \circ f= j\circ h \circ f$, the composition $h\circ f$ is the sought lifting of $j$.

 The last statement of the lemma follows easily from this characterization of projective $A$-sets.
 \end{proof}

\begin{proposition}\label{prop:DescriptionOfProjectiveAsets}
 Every projective $A$-set $P$ is of the form $\bigvee_{i\in I}e_iA$ where $e_i^2=e_i$ are idempotents in $A$.
\end{proposition}

\begin{proof}
 By Lemma \ref{lemma:ProjectiveASetsSplitsIntoFree}, we can assume that $P$ is an $A$-subset of some free $A$-set $\bigvee_{i\in I}A.x_i$ and that there is a morphism $f:\bigvee_{i\in I}A.x_i \to P$ that is the identity on $P$. Clearly, $P=\bigvee_{i\in I}(A.x_i\cap P)$. Let $i\in I$ such that $ A.x_i \cap P \neq \emptyset$.  Since  the composition
 $$ P \xrightarrow \vee_i Ax_i \xrightarrow{f}  P$$
 is the identity on $P$, we have for any $ a.x_i \in A.x_i \cap P$ that $f(a.x_i) = b.x_j$ implies $a.x_i = b.x_j$. In particular, this shows that $i=j$.

 For any $i\in I$ such that $A.x_i \cap P \neq \emptyset$, let $e_i\in A$ be an element such that $f(x_i)=e_i.x_i$. Since this map composed with the inclusion of $P \to \bigvee_{i\in I}A.x_i $ is identity, we have that $e_i.x_i=f(e_i.x_i)=e_i.f(x_i)=e_i^2.x_i$, and thus $e_i^2=e_i$.
\end{proof}

\begin{corollary}\label{coro:LocalizationAndBaseExtensionOfProjectivesAreProjectives} Let $P$ be a projective $A$-set and let $f:A\to B$ be a morphism of monoids.
\begin{enumerate}\item[1.] $P\otimes_A B$ is a projective $B$-set. In particular $S^{-1}P$ is a projective $S^{-1}A$-set if $S$ is a multiplicative subset of $A$.
\item[2.] $P_\Z$ is a projective $A_\Z$-module.
\end{enumerate}
\end{corollary}
\begin{proof} By Lemma \ref{lemma:ProjectiveASetsSplitsIntoFree} and Proposition \ref{prop:DescriptionOfProjectiveAsets}, we only need to consider the case that $P=eA$ for some idempotent $e$ in $A$. Direct computation shows that $eA\otimes_AB=f(e)B$, where $f(e)$ is an idempotent in $B$, and that $(eA)_\Z=e(A_\Z)$, where $e$ is regarded as an element in $A_\Z$ which is again idempotent.
\end{proof}

\begin{proposition}\label{prop:SplittingOfAdmissibleShortExactSequences} An admissible short exact sequence of $A$-sets
$$0\longrightarrow M \stackrel{i}{\longrightarrow} N \stackrel{j}{\longrightarrow} P \longrightarrow 0$$ is splitting exact if $P$ is projective. In other words, the sequence is isomorphic to the canonical short exact sequence
$$0\longrightarrow M \stackrel{i}{\longrightarrow} M\vee P \stackrel{j}{\longrightarrow} P \longrightarrow 0.$$
\end{proposition}
\begin{proof} Let $s:P \to N$ be the section of $j$. Then we have a morphism $M\vee P\stackrel{i\vee s}\longrightarrow N$. One checks that $i\vee s$ is an isomorphism and it gives the isomorphism of the two admissible short exact sequences.
\end{proof}

\subsubsection{Finitely generated  $A$-sets}

 An $A$-set $M$ is called \emph{finitely generated} if there exist finitely many elements $m_1, \cdots m_t$ such that $M$ is the union $$M=Am_1\cup\cdots\cup Am_t.$$

\begin{lemma}\label{lemma:MIsFGIffMZIsFG}
 $M$ is a finitely generated $A$-set if and only if $M_\mathbb{Z}$ is a finitely generated $A_\Z$-module.
\end{lemma}

\begin{proof}
Clearly $M_\mathbb{Z}$ is a finitely generated $A_\Z$-module if $M$ is a finitely generated $A$-set. Conversely, assume that $M_\mathbb{Z}$ is generated by $\{m_1, \cdots, m_n\}$ as an $A_\Z$-module. Since $M_\Z$ is a free abelian group on $M-\{*\}$, we can write each $m_i$ uniquely as $$m_i = \sum\limits_{j=1}^{r_i} n_{ij} m_{ij}$$ for some $ n_{ij}\in \mathbb{Z}$ and $m_{ij} \in M$, where $r_i$ is some integer depending on $i$.  Let $N$ be the $A$-subset generated by $\{m_{ij}\big\vert i=1,\cdots, n \mbox{ and } j =1,\cdots r_i\}$. Since $N_\Z$ equals $M_\Z$, we see that $N=M$ which shows that $M$ is finitely generated.
\end{proof}

An $A$-set $M$ is called \emph{Noetherian} if all the $A$-subsets of $M$ are  finitely generated. It follows immediately that if $M$ is a Noetherian $A$-set, then the $A$-subsets and quotients of $M$ are also Noetherian.

A monoid $A$ is called \emph{Noetherian} if and only it is finitely generated. It is well known (see \cite[Theorem 5.1]{Gilmer}) that the ideals of a Noetherian monoid $A$ are finitely generated. So a Noetherian monoid $A$ is also Noetherian as an $A$-set. But the converse is not true. For example, let $\Fun[G]$ denote the monoid associated to a group $G$ which is not finitely generated, then $\Fun[G]$ is not Noetherian as a monoid but it is obviously Noetherian as an $\Fun[G]$-set.

\begin{proposition}\label{prop:FGIffNoet}
 Let $A$ be a Noetherian monoid and $M$ be an $A$-module. Then $M$ is a finitely generated $A$-set if and only if $M$ is Noetherian.
\end{proposition}
\begin{proof} The if part is easy. To see the other implication, let $N$ be an $A$-subset of $M$, then $N_\Z$ is an $A_\Z$-submodule of $M_\Z$. $M_\Z$ is finitely generated as an $A_\Z$-module by Lemma \ref{lemma:MIsFGIffMZIsFG}. Since $A$ is finitely generated, $A_\Z$ is a Noetherian ring. So $M_\Z$ is Noetherian which implies that $N_\Z$ is also finitely generated as an $A_\Z$-module. By Lemma \ref{lemma:MIsFGIffMZIsFG}, this proves that $N$ is finitely generated as an $A$-set.
\end{proof}

\subsubsection{Examples}\label{A-set_examples}

The constructions of the previous sections provide already a variety of examples. Given a monoid $A$, there are the trivial $A$-set $0=\{\ast\}$, the product $A^n=\prod_{i=1}^n A$, the coproduct $A^{\vee n}=\bigvee_{i+1}^n A$, the tensor product $A^{\otimes n}=\bigotimes_{i=1}^n A$ and the smash product $A^{\wedge n}=\bigwedge_{i=1}^n A$. If $f:A\to B$ is a morphism of monoids, then $B$ is an $A$-set by defining $a.b=f(a)b$ for $a\in A$ and $b\in B$.

Any ideal $I$ of $A$ is an $A$-set since $ab\in I$ for every $a\in A$ and $b\in I$. Consequently, the quotient $A/I$ is also an $A$-set. In particular, $\Fun[A^\times]=A/(A-A^\times)$ and $\Fun[A^\int]=A/(A-A^\int)$ are $A$-sets.

The category of $\Fun$-sets is nothing else than the category of pointed sets together with morphisms that respect the base point.

More generally, let $A=\Fun[G]=G\amalg\{0\}$, where $G$ is a group and let $M$ be an $A$-set. Then the action of $A$ on $M$ restricts to an action of $G$ on $M$ and we see that $M$ decomposes into disjoint $G$ orbits. Let $\cC$ be the category whose objects are sets and whose morphism sets $\Hom(M,N)$ are the sets of $M\times N$-matrices (with $M$ and $N$ possibly being infinite) with coefficients in $A$, such that each row has at most one entry that differs from $0$. Define a functor $\cF:\cC\to A-\Mod$ by sending a set $S$ to $\bigvee_{s\in S} A.s$ and a morphism $f:S\to T$ to the morphism $\cF(f):\cF(S)\to\cF(T)$ of $A$-sets, which is defined by sending an element $a.s$ to $ag.t$ if there is  a non-trivial entry $g$ in the row corresponding to $s$ and the column corresponding to $t$, and to $\ast$ otherwise. Then $\cF$ is an equivalence of categories. Note that a morphism $f$ is normal if and only if the corresponding matrix has at most one non-trivial entry in each row. This builds a bridge to Haran's viewpoint on $\Fun$-geometry (\cite{Haran07}): the so-called \emph{$\F$-ring} corresponding to the monoid $A=\Fun[G]$ is the category of finitely generated $A$-sets together with all normal morphisms between them.

Let $A$ be the ``polynomial ring'' $\Fun[T]=\{T^i\}_{i\geq 0}\cup\{0\}$. An $\Fun[T]$-set $M$ is characterized by the base point preserving map $T:M\to M$ that sends $m$ to $T.m$. On the other hand, every base point preserving map $T:M\to M$ of a pointed set defines an action of $\Fun[T]$ on $M$ and gives $M$ the structure of an $\Fun[T]$-set. We describe certain $\Fun[T]$-sets in more detail.

The prime ideals of $\Fun[T]$ are $(0)=\{0\}$ and $(T)=\{T^i\}_{i\geq1}\cup\{0\}$. All other ideals are of the form $(T^k)=\{T^i\}_{i\geq k}\cup\{0\}$ for some $k\geq 0$. All the ideals $(T^k)$ are isomorphic to $\Fun[T]$ as an $\Fun[T]$-set. But the quotient $\Fun[T]/(T^k)$ is an $\Fun[T]$-set with $k+1$ elements $\{\ast,T^0,T^1,\ldots,T^{k-1}\}$. The base point preserving map $T:\Fun[T]/(T^k)\to \Fun[T]/(T^k)$ is given by $T.T^i=T^{i+1}$ for $0\leq i<k-1$ and $T.T^{k-1}=\ast$. The base extension of $\Fun[T]/(T^k)$ is $\Z[T]/(T^k)$.

Another family of $\Fun[T]$-sets are the pointed sets $\{\ast,T^0,\dotsc,T^{k-1}\}$ with $T.T^i=T^{i+1}$ for $0\leq i<k-1$ and $T.T^{k-1}=T^l$ for some $l\in\{0,\dotsc, k-1\}$. Its base extension to $\Z$ is $\Z[T]/(T^k-T^l)$.

More generally, every finitely generated $\Fun[T]$-set can be described as follows. Let $M$ be generated by $\{m_0=\ast,m_1,\dotsc,m_r\}$. Then we have for every $i\in\{1,\dotsc,r\}$ either that $\{T^j.m_i\}_{j\geq0}$ is a infinite set that is disjoint with $M_{< i}=\bigcup_{l=0}^{i-1} \{T^j.m_l\}_{j\geq0}$, or that there is a relation $T^{k_i}.m_i = T^{k'_i}.m_{l(i)}$ for some $T^{k'_i}.m_{l(i)}\in M_{<i}\cup\{T^0.m_i,\dotsc,T^{k_i-1}.m_i\}$. These relations describe $M$ completely. The base extension of $M$ to $\Z$ is isomorphic to $\Z[T_1]\oplus\dotsb\oplus\Z[T_r]/\langle T_i^{k_i}-T_{l(i)}^{k'_i}\rangle$ (with the convention $T_0=0$).

Another interesting $\Fun[T]$-set is the monoid $\Fun[T,T^{-1}]$ with the $\Fun[T]$-set structure given by the canonical inclusion $\Fun[T]\to\Fun[T,T^{-1}]$. This $\Fun[T]$-set is not finitely generated, but it is an injective object in $\Fun[T]-\Mod$. So are the finite coproducts $\Fun[T,T^{-1}]^{\vee n}$.


\section{The geometry of monoids}\label{geometry_of_monoids}

After recalling the notion of $\Mo$-schemes as defined by Connes and Consani (\cite{CC09}), following the ideas of Kato (\cite{Kato94}) and Deitmar (\cite{Deitmar05}), we develop a theory of $\cO_X$-modules and quasi-coherent sheaves for $\Mo$-schemes based on our notion of $A$-sets from the previous section. One may regard this section as a continuation of \cite{Deitmar05}.

 The theories of $\cO_X$-modules for $\Mo$-schemes and for usual schemes are analogous to a large extend. We forgo proofs when they are in complete analogy to usual scheme theory. At points where the theory differs, we will provide detailed explanation.

\subsection{$\Mo$-schemes}

In this subsection,  we recall the theory of $\Mo$-schemes from  cf.\ \cite{CC09} and \cite{Deitmar05}.

\subsubsection{Definition and general properties}

A \emph{monoidal space} is a pair ($X$, $\mathcal{O}_X$) consisting of a topological space $X$ and a sheaf of monoids $\mathcal{O}_X$, called the \emph{structure sheaf}. If ($X$, $\mathcal{O}_X$) is a monoidal space and $f: X \to Y$ is a continuous map of topological spaces, $f_*\mathcal{O}_X$ is a sheaf of monoids on $Y$. A \emph{morphism of monoidal spaces} $(X, \mathcal{O}_X)\to (Y, \mathcal{O}_Y)$ is a pair $(f, f^\#)$ consisting of a continuous map $f: X\to Y$ of topological spaces and a morphism of sheaves $f^\#: \mathcal{O}_Y \to f_* \mathcal{O}_{X}$.

A morphism $f:A\to B$ of monoids is \emph{local} if $f^{-1}(B-B^\times)=A-A^\times$. The \emph{stalk $\cO_{X,x}$ of $\cO_X$ at $x$}, i.e.\ the colimit $\colim \cO_X(U)$ over all open neighborhoods $U$ of $x$, always exists by Proposition \ref{monoid-limits}. We say that a morphism $(f, f^\#): (X,\cO_X)\to (Y,\cO_Y)$ of monoidal spaces is \emph{local} if for all $x\in X$, the morphism $f_x^\#: \cO_{Y,f(x)}\to\cO_{X,x}$ between stalks is local.

Let $A$ be a monoid. Recall from section \ref{localization_of_monoids} that $D_f$ is the set of all prime ideals of $A$ that do not contain $f\in A$. The \emph{spectrum of a monoid $A$} is the set $\spec A$ of all prime ideals of $A$ endowed with the topology generated by $\{D_f\}_{f\in A}$. By Lemma \ref{D_f-topology}, the family $\{D_f\}$ forms a basis for this topology. The structure sheaf $\cO_{\spec A}$ is defined by $\cO_{\spec A}(D_f)=A_f$ for all $f\in A$.

An \emph{affine $\Mo$-scheme} is a monoidal space that is isomorphic to the monoidal space ($\spec A$, $\cO_{\spec A }$) for some monoid $A$.  An \emph{$\Mo$-scheme} is a monoidal space that admits an affine cover, i.e.\ an open cover by affine $\Mo$-schemes. A \emph{morphism of $\Mo$-schemes} is a local morphism of monoidal spaces.

Let $f:A\to B$ be a morphism of monoids. The inverse image of a prime ideal of $B$ is a prime ideal of $A$. As in the case of usual schemes, this yields a continuous map $\varphi:\spec B\to\spec A$ and a morphism $\varphi^\#:\cO_{\spec A}\to\cO_{\spec B}$ of structure sheaves such that the pair $(\varphi,\varphi^\#)$ is a local morphism of monoidal spaces. Conversely, taking global section of $\varphi$ gives back $f$. More precisely, $\Mo$ is dual to the category of affine $\Mo$-schemes.

Let $\cB$ be the set of all affine open subsets of $X$. Then, by Lemma \ref{D_f-topology}, $\cB$ forms a basis of the topology of $X$. As in usual scheme theory, we have that for $X=\Spec A$ and a point $x=\p$ of $X$,   $\cO_{X,x}\simeq A_\p$.

\begin{proposition}
 The category of $\Mo$-schemes contains finite limits.
\end{proposition}

\begin{proof}
 It is enough to prove that finite fiber product of $\Mo$-schemes exists, which is proven in \cite[Proposition 3.1]{Deitmar05}.
\end{proof}

An $\Mo$-scheme is \emph{integral} if it can be covered by affine schemes that are isomorphic to the spectrum of an integral monoid. If an $\Mo$-scheme is integral, then every affine open subset is the spectrum of an integral monoid.

\subsubsection{Base extension to $\Z$}

The {\it base extension of $\Spec A$ to $\Z$} is $\Spec A_\Z$. Note that this extends to a functor $-\otimes_\Fun\Z$ from the category of affine $\Mo$-schemes to the category of affine schemes. Let $X$ be an $\Mo$-scheme. The intersection of two affine open subschemes of $X$ is an affine open subscheme by the definition of the topology of an affine $\Mo$-scheme and Lemma \ref{D_f-topology}.

Consider an affine open cover that is closed under intersections. Together with the inclusions of subsets, this defines a directed system, and $X$ is the colimit over this directed system. We define $X_\Z$ as the colimit over the base extension of the directed system to $\Z$. One can show that $X_\Z$ does not depend on the choice of cover. The association $X\mapsto X_\Z$ extends to a functor $-\otimes_\Fun\Z$ from the category of $\Mo$-scheme to the category of schemes, see \cite{Deitmar05} for details.

Since an open inclusion $\iota:V\hookrightarrow U$ of affine $\Mo$-schemes $U=\spec A$ and $V=\spec B$ means that $B$ is the localization of $A$ at some multiplicative subset $S\subset A$, the base extension $\iota_\Z: V_\Z\to U_\Z$ is induced by the localization $A_\Z\to S^{-1}A_\Z$ (cf.\ Lemma \ref{localization_and_base_extension}) and thus injective. Since the base extension $X_\Z$ of an $\Mo$-scheme is defined as the colimit over a system of inclusions, all the canonical morphisms $U_\Z\to X_\Z$ are injective where $U_\Z=\Spec A_\Z$ is the base extension of an affine open subset  $U=\spec A$ of $X$. If $U$ is any open subset of $X$, then we define $U_\Z$ as the union of all base extensions of affine open subsets of $U$ inside $X_\Z$.

This association comes indeed from a continuous map $\beta: X_\Z\to X$. Let $x$ be a point of $X_\Z$ and $U_\Z=\Spec A_\Z$ an affine open neighborhood that is the base extension of an affine open subset $U=\spec A$ of $X$. Then $x=\p$ is a prime ideal of $A_\Z$, and it is immediately verified that $\q=\p\cap A$ is a prime ideal of $A\subset A_\Z$. This prime ideal $\q$ of $A$ defines a point $y$ of $U\subset X$. We define $\beta(x)=y$. To verify that this is independent of the choice of $U$, we let $V_\Z$ be another affine neighborhood of $x$, which we can assume to be a subset of $U_\Z$ by replacing $V$ with $V\cap U$. Then $V=\spec S^{-1}A$ for some multiplicative subset $S$ of $A$. Let $f:A\to S^{-1}A$ be the canonical map. The independence of $\beta(x)$ from the choice of $U$ follows from the equality
\begin{multline*}
 f_\Z^{-1}(\p)\cap A \ = \ \{a\in A\mid f_\Z(a)\in\p\} \\ = \ \{a\in A\mid f(a)\in\p\cap S^{-1}A\} \ = \ f^{-1}(\p\cap S^{-1}A).
\end{multline*}
for any prime ideal $\p$ of $S^{-1}A$.

\begin{theorem}
  The map $\beta: X_\Z\to X$ is continuous and the inverse image of an open subset $U$ of $X$ is $U_\Z$. The map $\beta$ is functorial in $X$.
\end{theorem}

\begin{proof}
 To show that $\beta$ is continuous, it is enough to show that $\beta^{-1}(U)=U_\Z$ for open subsets $U$ of $X$. By the definition of $U_\Z$, it is enough to verify this for affine open subsets $U=\spec A$.

 If $x$ is in $U_\Z$, then, by definition of $\beta$, the image $\beta(x)$ is in $U$. If $x$ is not in $U_\Z$, but in another affine neighbourhood $V_\Z$, then we have to show that $\beta(x)\notin V\cap U$. If $V=\spec A$, then $x=\p$ is a prime ideal of $A$ and $V\cap U=\spec S^{-1}A$ for some multiplicative subset $S$ of $A$ (cf.\ Lemma \ref{D_f-topology}). That $x\notin (V\cap U)_\Z$ means that $p\cap S$ is not empty. Since $S\subset A$, this implies that $(\p\cap A) \cap S$ is not empty, and thus $\beta(\p)\notin S^{-1}A$ (cf.\ Lemma \ref{prime_ideal_in_U_S}).

 The functoriality of $\beta: X_\Z\to X$ follows from the local definition of $\beta$ and the commutativity of the diagram
 $$\xymatrix{A\ar[rr]^{f}\arincl{[d]}&&B\arincl{[d]}\\ A_\Z\ar[rr]^{f_\Z}&& B_\Z} $$
 for any morphism $f:A\to B$ of monoids.
\end{proof}

The theorem yields the following consequence, which follows from a general property of continuous maps.

\begin{corollary}
 Let $\{U_i\}$ be a family of open subsets of $X$. Then $(U_i\cap U_j)_\Z=U_{i,\Z}\cap U_{j,\Z}$ and $(\bigcup U_i)_\Z=\bigcup U_{i,\Z}$.\qed
\end{corollary}

An $\Mo$-scheme $X$ is \emph{separated} if $X_\Z$ is a separated scheme.

\begin{remark}
 We use this indirect definition because the usual definition that the diagonal $\Delta(X)$ is closed in $X\times X$ does not produce a good notion of separatedness for $\Mo$-schemes: Both the projective line and the affine double line over $\Fun$ can be covered by two affine lines that intersect in a multiplicative group scheme. As explained in section \ref{Examples_of_Mo-schemes} in case of the projective line, this topological space consists of three points: two closed points and one generic point. This determines the topology completely.

 However, we conjecture that the following condition on $X$ is equivalent to separatedness: for all points $x$ and $y$ in $X$ and all common generalizations $z$ of $x$ and $y$, we have that $x=y$ if the images of the maps $\cO_{X,x}\to\cO_{X,z}$ and $\cO_{X,x}\to\cO_{X,z}$ are equal.
\end{remark}

\subsubsection{$\Mo$-schemes of finite type}

An $\Mo$-scheme is \emph{locally of finite type} if it can be covered by affine schemes that are isomorphic to the spectrum of a finitely generated monoid. This property is local, i.e.\ an $\Mo$-scheme is locally of finite type if and only if every affine open subscheme is isomorphic to a finitely generated monoid. An $\Mo$-scheme is \emph{of finite type} or \emph{Noetherian} if it is locally of finite type and quasi-compact.

The following is \cite[Lemma 2]{Deitmar06}.

\begin{lemma}
 An $\Mo$-scheme $X$ is of finite type if and only if $X_\Z$ is a scheme of finite type.\qed
\end{lemma}

Since an $\Mo$-scheme of finite type is covered by finitely many affine open subschemes, Lemma \ref{finitely_many_prime_ideals} implies the following fact.

\begin{lemma}
 An $\Mo$-scheme of finite type consists of finitely many points.\qed
\end{lemma}

The stalks of an $\Mo$-scheme that is locally of finite type have a particularly simple form in contrast to the theory of schemes.

\begin{proposition}
 Let $X$ be an $\Mo$-scheme that is locally of finite type. Every point $x\in X$ has an open neighborhood $U$ such that $\cO_{X,x}\simeq\cO_X(U)$.
\end{proposition}

\begin{proof}
 Let $x\in X$. Then there is an affine open neighborhood $V$ of $x$, i.e.\ $V\simeq\spec A$ for a finitely generated monoid $A$. This means that $\p=x$ is a prime ideal of $A$. By the definition of a prime ideal, $S=A-\p$ is a multiplicative set, and by Corollary \ref{maximal_ideal_in_U_S}, $U_S$ has a unique maximal element, namely, $\p$. This means that $U_S$ is contained in all sets of the form $D_f$ with $f\notin\p$. By Proposition \ref{U_S=D_f} there is indeed an $f\in A$ such that $U_S=D_f$, and thus $\cO_{X,x}\simeq\cO_X(D_f)$.
\end{proof}

The previous proposition together with the fact that $\m_A=A-A^\times$ is the unique maximal ideal of $A$ implies the following statement, which marks a major simplification to usual scheme theory.

\begin{corollary}
 Let $X$ be locally of finite type and let $\cB$ be the set of all open affine subsets of $X$. The association $x\mapsto \bigcap U$ where $U$ runs through all open neighborhood of $x$ in $X$ defines a bijection $X\to\cB$. Its inverse map sends an affine open subset $U=\spec A$ of $X$ to the maximal ideal $\m_A$ of $A$, which is a point of $U\subset X$.\qed
\end{corollary}

\subsubsection{Examples}
\label{Examples_of_Mo-schemes}

The spectrum of $\Fun=\{0,1\}$ consists of precisely one point, namely, the unique prime ideal $\{0\}$ of $\Fun$. The stalk at $\{0\}$ is equal to $\Fun$. The base extension to $\Z$ is $\Spec \Z$. The $\Mo$-scheme $\spec \Fun$ is a terminal object in the category of $\Mo$-schemes.

More generally, let $G$ be a group and $A=\Fun[G]=\{0\}\amalg G$. Then $\spec A$ consists of the unique prime ideal $\{0\}$ of $A$ and the stalk at $\{0\}$ is $A$. The base extension to $\Z$ is $\Spec A_\Z=\Spec\Z[G]$.

In particular, if $G$ is a free abelian group on $n$ generators $T_1,\dots,T_n$, then $A=\Fun[T_1^{\pm 1},\dotsc,T_n^{\pm 1}]$ and $A_\Z = \Z[T_1^{\pm1},\dotsc,T_n^{\pm1}]$ and thus $(\spec A)_\Z\simeq\Gm^n$. This justifies denoting $\spec A$ by $\G_{m,\Fun}^n$.

Let $A=\Fun[T_1,\dotsc,T_n]$ be the free monoid on $n$ generators $T_1,\dots,T_n$. Then $A_\Z = \Z[T_1,\dotsc,T_n]$ and thus $(\spec A)_\Z\simeq\A^n$. This justifies to denote $\spec \Fun[T_1,\dotsc,T_n]$ by $\A_\Fun^n$ and call it the \emph{$n$-dimensional affine space over $\Fun$}. The prime ideals of $A$ are of the form $\p_I=\bigcup_{i\in I}T_iA$ where $I$ is a subset of $\{1,\dotsc,n\}$ and $T_iA=\{T_ia\mid a\in A\}$. The stalk of the structure sheaf at $\p_I$ is the localization of $A$ at the multiplicative set $S$ that contains all products of elements $T_j$ where $j\notin I$.

Let $U_1=\spec\Fun[T_1]$, $U_2=\spec\Fun[T_2]$ and $U=\spec\Fun[T^{\pm1}]$, i.e.\ $U_1\simeq\A^1_\Fun\simeq U_2$ and $U\simeq\G_{m,\Fun}$. The monoid morphisms $\Fun[T_1]\to\Fun[T^{\pm1}]$ defined by $T_1\mapsto T$ and $\Fun[T_2]\to\Fun[T^{\pm1}]$ defined by $T_2\mapsto T^{-1}$ induce morphisms $U\to U_1$ and $U\to U_2$ of affine schemes. These morphisms are open inclusions that send the unique point of $U$ to the generic points of $U_1$ and $U_2$, respectively. The colimit over these three affine $\Mo$-schemes together with these two morphisms defines an $\Mo$-scheme with one generic point $\eta$ and two closed points $\p_1$ and $\p_2$. The stalk at $\eta$ is isomorphic to $\Fun[T^{\pm1}]$ and the stalk at $\p_i$ is isomorphic to $\Fun[T_i]$ for $i=1,2$. The base extension to $\Z$ is isomorphic to the projective line $\P^1$. This justifies to denote this $\Mo$-scheme by $\P^1_\Fun$ and call it the \emph{projective line over $\Fun$}. Similarly one defines the \emph{$n$-dimensional projective space $\P_\Fun^n$ over $\Fun$}.

This sort of construction generalizes to all toric varieties. In \cite{Deitmar07}, Deitmar proves that the class of separated, connected, integral schemes of finite type that are base extensions of $\Mo$-schemes to $\Z$ is the class of toric varieties (also cf.\ \cite{LL09b}).


\subsection{$\cO_X$-modules}

In this section we set up the theory of $\cO_X$-modules. Like the theory of $A$-sets was similar to the theory of modules over a ring, the theory of $\cO_X$-modules is similar to the usual theory for schemes. A difference is marked, again, by the special class of normal morphisms.

\subsubsection{Definition and general properties}

Let $X$ be an $\Mo$-scheme and $\cO$ be a sheaf of monoids on $X$. An \emph{$\cO$-module} is a sheaf $\cM$ of pointed sets where $\cM(U)$ is an $\cO(U)$-set for all opens $U$ of $X$ such that the restriction maps $\cM(V)\to\cM(U)$ are compatible with $\cO(V)\to\cO(U)$ for all opens $V\subset U$. A \emph{morphism of $\cO$-modules $
\varphi:\cM\to\cN$} is a morphism of sheaves such that for all open subsets $U\subset X$, the map $\cM(U)\to\cN(U)$ is a morphism of $\cO(U)$-sets. We denote the category of
$\cO$-modules by $\cO-\Mod$.

In particular, if $\cO=\cO_X$ is the structure sheaf of $X$, then we have defined the notion of an $\cO_X$-module.

The $\cO_X$-module $\0$ that associates to every open subset $U$ of $X$ the trivial $\cO_X(U)$-set $\{\ast\}$ is an initial and a terminal object in $\cO_X-\Mod$. Consequently, there is for all $\cO_X$-modules $\cM$ and $\cN$ a unique morphism $0:\cM\to\cN$ from $\cM$ to $\cN$ that factors through $\0$, which makes $\Hom(\cM,\cN)$ a pointed set, or $\Fun$-set. The association $\Hom(\cM,\cN)(U)=\Hom_{\cO\vert_U}(\cM\vert_U,\cN\vert_U)$ for open subsets $U$ of $X$ gives $\Hom(\cM,\cN)$ the structure of an $\cO_X$-module. This is functorial in both $\cM$ and $\cN$, thus $\Hom(-,-)$ is a bifunctor from $\cO_X-\Mod$ into itself.

The stalk $\cM_x$ of an $\cO_X$-module $\cM$ at $x\in X$ is naturally an $\cO_{X,x}$-set. Taking stalks is functorial in $\cM$, i.e.\ a morphism $f:\cM\to\cN$ of $\cO_X$-modules yields morphisms $f_x:\cM_x\to\cN_x$ of $\cO_{X,x}$-sets for every $x\in X$. Conversely, $f$ is determined by all morphisms $f_x$ between the stalks. We will frequently use the fact that, given a \emph{presheaf of $\cO_X$-modules}, i.e.\ a presheaf that satisfies all properties of an $\cO_X$-module except for the sheaf axiom, then its sheafification is naturally an $\cO_X$-module. The image of a morphism $f:\cM\to \cN$ of $\cO_X$-modules is the sheafification of the presheaf $\im f$ that associates to an open $U$ of $X$ the $\cO_X(U)$-set $f(\cM(U))$.

As in usual scheme theory,  we define the (co)kernel of a morphism $f:\cM\to\cN$ as the (co)equalizer of $f$ and $0:\cM\to\cN$. If $\cN\subset\cM$ is a sub-$\cO_X$-module, the quotient $\cO_X$-module $\cM/\cN$ is the cokernel of the inclusion morphism $\cN\hookrightarrow\cM$. A diagram $\cM_1\stackrel f\longrightarrow \cM_2\stackrel g\longrightarrow \cM_3$ is exact at $\cM_2$ if $\ker(g)=\im(f)$. A short exact sequence of $\cO_X$-modules is a sequence
$$ \0 \longrightarrow \cM_1 \longrightarrow \cM_2 \longrightarrow \cM_3 \longrightarrow \0 $$
that is exact at $\cM_1$, $\cM_2$ and $\cM_3$.

\subsubsection{Normal morphisms}

As in the case of $A$-sets where $A$ is a monoid, the category of $\cO_X$-modules contains epimorphisms that are not normal, i.e.\ not a cokernel. This leads to the following definition.

A morphism $f:\cM\to\cN$ of $\cO_X$-modules is \emph{normal} if $f_x:\cM_x\to\cN_x$ is normal for all $x\in X$.  The following is derived by employing Proposition \ref{normal_lemma} to the definition
.
\begin{proposition}
 \label{prop:NmMrphsmAreClosedUnderComposition}
 Normal morphisms of $\cO_X$-modules are closed under compositions.\qed
\end{proposition}
\begin{definition}
 We call a short exact sequence of $\cO_X$-modules
 $$ \0 \longrightarrow \cM_1 \longrightarrow \cM_2 \longrightarrow \cM_3 \longrightarrow \0 $$
 \emph{admissible} if and only if all morphisms in the sequence are normal.
\end{definition}

\subsubsection{Tensor products}\label{sec:TensorProductOnMoSch}

Given two $\cO_X$-modules $\cM$ and $\cN$, their \emph{tensor product} $\cM\otimes_{\cO_X}\cN$ is the sheafification of the presheaf of $\cO_X$-modules sending $U$ to
$\cM(U)\otimes_{\cO_X(U)}\cN(U)$. It is functorial in both $\cM$ and $\cN$. It satisfies the universal property of the tensor product, and the functor $\cM\otimes_{\cO_X}(-)$ is left-adjoint to $\Hom(\cM,-)$.

We define the \emph{smash product} $\cM\wedge\cN$ of $\cM$ and $\cN$ as the sheafification of the presheaf of $\cO_X$-modules sending $U$ to $\cM(U)\wedge\cN(U)$.

Given a morphism $f:Y\to X$ of $\Mo$-schemes and an $\cO_Y$-module $\cN$, then the \emph{direct image sheaf $f_\ast\cN$}, which maps open subsets $U$ of $X$ to $
\cN(f^{-1}(U))$, carries naturally the structure of an $\cO_X$-module. If $\cM$ is an $\cO_X$-module, then we define $f^{-1}\cM$ as the sheafification of the presheaf on $Y$ sending open subsets $U$ to $\colim\cM(V)$ where the colimit is taken over all open subsets $V$ of $X$ that contain $f(U)$. By Proposition \ref{monoid-limits}, the colimit of a directed diagram of monoids exists. Thus $f^{-1}\cO_X$ is not merely a sheaf, but a sheaf of monoids on $Y$, and for any other $\cO_X$-module $\cM$, the sheaf $f^{-1}\cM$ is an $f^{-1}\cO_X$-module. There is a canonical morphism $f^{-1}\cO_X\to\cO_Y$ such that for all open subsets $U$ of $Y$, the map $\cO_X(U)\to\cO_Y(U)$ is a monoid morphism. Thus we can regard $\cO_Y$ as $f^{-1}\cO_X$-module. We define the \emph{inverse image sheaf $f^\ast\cM$}  as the sheafification of the presheaf of $\cO_Y$-modules $f^{-1}\cM\otimes_{f^{-1}\cO_X}\cO_Y$. All these constructions $f_\ast$, $f^{-1}$ and $f^\ast$ are functorial. We call $f^\ast:\cO_X-\Mod\longrightarrow\cO_Y-\Mod$ also the \emph{base change functor from $X$ to $Y$ (along $f$)}.

\subsubsection{Base extension to $\Z$}

Formally similar to the base change functor $f^\ast$ from the previous section, we define the base extension to $\Z$ as the base extension functor $\beta^\ast$ along $\beta:X_\Z\to X$. We make this precise.

The sheaf $f^{-1}\cM$ is defined for any sheaf $\cM$ on any topological space $X$ and for any continuous map $f:Y\to X$ of topological spaces. If $X$ is an $\Mo$-scheme and $\cM$ an $\cO_X$-module, then $f^{-1}\cO_X$ is a sheaf of monoids on $Y$ and $f^{-1}\cM$ is an $f^{-1}\cO_X$-module on $Y$. In the case of our interest where $f=\beta:X_\Z\to X$, we define the \emph{base extension $\cM_\Z$ of $\cM$ to $\Z$} as the tensor product $\beta^{-1}\cM\otimes_{\beta^{-1}\cO_X}\cO_{X_\Z}$, which is the sheafification of the presheaf on $X_\Z$ that sends an open $U$ to $\beta^{-1}\cM(U)\otimes_{\beta^{-1}\cO_X(U)} \cO_{X_\Z}(U)$. This is functorial in $\cM$ and defines the \emph{base extension functor $-\otimes_{\cO_X}\cO_{X_\Z}:\cO_X-\Mod\to \cO_{X_\Z}-\Mod$} where $\cO_{X_\Z}$ is the structure sheaf of $X_\Z$.

Note that for any open set $U$ of $X$ and for any $\cO_X$-module $\cM$, we have
\begin{multline*} \cM_\Z(U_\Z) \ \simeq \ \colim \cM(V)\otimes_{\colim\cO_X(V)}\cO_{X_\Z}(U_\Z) \\ \simeq \ \cM(U)\otimes_{\cO_X(U)}\cO_X(U)_\Z \ \simeq \ \cM(U)_\Z  \end{multline*}
where the colimit is taken over the system of all open subsets $V$ of $X$ such that $U=\beta(U_\Z)\subset V$, which has $U$ as initial object. In particular, we have $\cO_{X_\Z}\simeq(\cO_X)_\Z$.

Similarly as for $A$-sets, the following properties are easy to prove.

\begin{lemma} \label{base_extension_properties_for_sheaves}
Let $X$ be an $\Mo$-scheme with structure sheaf $\cO_X$ and let $f:\cM\to \cN$ be a morphism of $\cO_X$-modules.
\begin{enumerate}
 \item We have $(\coker f)_\Z\simeq\coker f_\Z$ and $(\im f)_\Z\simeq\im f_\Z$.
 \item We have $(\cM\otimes_{\cO_X}\cN)_\Z\simeq \cM_\Z\otimes_{\cO_{X_\Z}}\cN_\Z$ and $(\cM\oplus \cN)_\Z\simeq \cM_\Z\oplus \cN_\Z$.
 \item If $Y\to X$ is a morphism of $\Mo$-schemes, then $(\cM\otimes_{\cO_X}\cO_Y)_\Z\simeq \cM_\Z\otimes_{\cO_{X_\Z}}\cO_{Y_\Z}$ as $\cO_{Y_\Z}$-modules.
 \item A morphism $f:\cM\to \cN$ of $\cO_X$-modules is a monomorphism (epimorphism) if and only if $f_\Z: M_\Z\to N_\Z$ is a monomorphism (epimorphism).\qed
\end{enumerate}
\end{lemma}

\subsection{Quasi-coherent sheaves}

If $M$ is an $A$-set and $X=\spec A$, then the association $U_S\mapsto S^{-1}M$ for multiplicative subsets $S \subset A$ defines an $\cO_X$-module $\widetilde M$. A morphism $\varphi:M\to N$ defines a morphism $\widetilde\varphi:\widetilde M\to\widetilde N$ of $\cO_X$-modules.  This yields a functor from $A-\Mod$ to the category of $\cO_X$-modules. If $\p$ is a prime ideal of $A$, then $M_\p$ is an $\cO_{X,\p}$-set.

Let $X$ be an $\Mo$-scheme with structure sheaf $\cO_X$. A \emph{quasi-coherent sheaf on $X$} is an $\cO_X$-module such that there is an affine open cover $\{U_i\}_{i\in I}$ of $X$ and for every $i\in I$, an $\cO_X(U_i)$-set $M_i$ such that $\cM\vert_{U_i}\simeq\widetilde M_i$ as $\cO_{U_i}$-modules. A \emph{morphism of quasi-coherent sheaves} is a morphism of $\cO_X$-modules. We denote the category of quasi-coherent sheaves on $X$ by $\QCoh X$. A \emph{coherent sheaf on $X$} is an $\cO_X$-module such that there is an affine open cover $\{U_i\}_{i\in I}$ of $X$ and for every $i\in I$, a finitely generated $\cO_X(U_i)$-set $M_i$ such that $\cM\vert_{U_i}\simeq\widetilde M_i$ as $\cO_{U_i}$-modules. The category $\Coh X$ is defined as the full subcategory of $\QCoh X$ whose objects are coherent sheaves. We refer to \cite{Deitmar05} for details.

As in usual scheme theory, we obtain the following fact (cf. the proof of \cite[Ch.\ II, Lemma 5.3 and Prop.\ 5.4]{Hartshorne}, which transfers \textit{mutatis mutandis} to our situation).

\begin{theorem}\label{quasi-coherent_any_cover}
 Let $X$ be an $\Mo$-scheme with structure sheaf $\cO_X$. An $\cO_X$-module $\cM$ is quasi-coherent if and only if for every affine open $U$ of $X$, there is an $\cO_X(U)$-set $M$ such that $\cM\vert_U\simeq\widetilde M$ as $\cO_U$-modules.\qed
\end{theorem}

\begin{corollary}\label{Cor_equiv_A-Mod_and_Coh_X}
 Let $A$ be a monoid and $X=\spec A$. Then the functor sending $M$ to $\widetilde M$ is an equivalence between the category of $A$-sets and the category of quasi-coherent sheaves on $X$. \qed
\end{corollary}

\begin{corollary}\label{corollary_char_q-coh}
 An $\cO_X$-module $\cM$ is quasi-coherent if and only if for all affine open subsets $V\subset U\subset X$, the restriction map $\res_{U,V}:\cM(U)\to\cM(V)$ extends to an isomorphism $S^{-1}\cM(U)\stackrel\sim\longrightarrow\cM(V)$ where $S$ is a multiplicative subset of $\cO_X(U)$ such that $S^{-1}\cO_X(U)\simeq\cO_X(V)$.\qed
\end{corollary}

Again, as in usual scheme theory (cf.\ \cite[Ch.\ II, Prop.\ 5.4]{Hartshorne}), we obtain the following fact.

\begin{theorem} \label{coherent_on_Noetherian_scheme}
 Let $X$ be a Noetherian $\Mo$-scheme with structure sheaf $\cO_X$.  An $\cO_X$-module $\cM$ is coherent if and only if for every affine open $U$ of $X$, there is a  finitely generated $\cO_X(U)$-set $M$ such that $\cM\vert_U\simeq\widetilde M$ as $\cO_U$-modules.\qed
\end{theorem}

\subsubsection{Limits and colimits}

The subcategories $\QCoh X$ and $\Coh X$ of $\cO_X-\Mod$ allow a series of categorical constructions that are compatible with the inclusion to $\cO_X-\Mod$. Namely, the trivial sheaf and the structure sheaf are both coherent sheaves. The homomorphism sheaf $\Hom(\cM,\cN)$ is (quasi-)coherent if $\cM$ and $\cN$ are both (quasi-)coherent. Since $\QCoh X$ and $\Coh X$ are full subcategories of $\cO_X-\Mod$, a (co)limit of a diagram of (quasi-)coherent sheaves in $\cO_X-\Mod$ is a (co)limit in the category of (quasi-)coherent sheaves provided the (co)limit is (quasi-)coherent.

\begin{lemma}
 The category $\QCoh X$ contains finite limits and small colimits. If $X$ is Noetherian, the category $\Coh X$ contains finite limits and finite colimits.
\end{lemma}

\begin{proof}
 Due to the local nature of sheaves, there is for every (co)limit $\cG$ of a diagram $\cD=\{\cF\}$ of $\cO_X$-modules an affine cover of $X$ such that for every $U$ in that cover, $\cG(U)$ is the (co)limit of the diagram of $\cO_X(U)$-sets $\cD(U)=\{\cF(U)\}$. By Proposition \ref{prop_localizations_and_limits}, finite limits and small colimits commute with localizations, and thus, by Corollary \ref{corollary_char_q-coh}, a finite limit resp.\ small colimit $\cG$ is quasi-coherent if all $\cF$ in $\cD$ are quasi-coherent. Note that finite (co)limits of finitely generated $\cO_X(U)$-sets are finitely generated if $\cO_X(U)$ is Noetherian, which solves the case for coherent sheaves.
\end{proof}

As a consequence, (co)kernels and finite (co)products of (quasi-)coherent sheaves are (quasi-)coherent.

\begin{remark}
 We see that the categories $\QCoh X$ and $\Coh X$ are as close to an abelian category as $A-\Mod$ is. Namely, all statements of Remark \ref{remark_abelian_category} apply \textit{mutatis mutandis} to the category of (quasi-)coherent sheaves.
\end{remark}

\subsubsection{Normal morphisms}

The category of quasi-coherent sheaves admits a characterization of normal morphisms in terms of coverings.

\begin{proposition}
 A morphism $f:\cM\to\cN$ of quasi-coherent sheaves is normal if and only if there is an affine cover $\{U_i\}_{i\in I}$ of $X$ such that for all $i\in I$, the morphism $f(U_i):\cM(U_i)\to\cN(U_i)$ is a normal morphism of $\cO_X(U_i)$-sets.
\end{proposition}

\begin{proof}
 By Lemma \ref{normal_and_localization}, the localization of a normal morphism is normal. Thus all localizations $f(U)$ where $U$ is affine and $U\subset U_i$ for some $i\in I$ are normal and those $U$ form a basis for the topology of $X$. Therefore all morphisms between the stalks are normal. If conversely, one of the restrictions $f(U_i)$ is not normal, then $f$ is not normal since $f(U_i)=f_x:\cM_x\to\cN_x$ where $x$ is the maximal point of $U_i$.
\end{proof}

\subsubsection{Base extension to $\Z$}

Let $X$ be an $\Mo$-scheme and $\beta:X_\Z\to X$ the canonical map. Let $\cM$ be an $\cO_X$-module. Recall that we defined the base extension of $\cM$ to $\Z$ as $\cM_\Z = \beta^{-1}\cM\otimes_{\beta^{-1}\cO_X}\cO_{X_\Z}$. Aim of this section is to show that the base extension of a quasi-coherent sheaf is quasi-coherent and, provided $X$ is Noetherian, that the base extension of a coherent sheaf is coherent.

We shall need the following lemma.

\begin{lemma}\label{lemma_for_base_ext_of_q-coh_sheaves}
 Let $\cM\in\QCoh X$. Let $U$ be an affine open subset of $X_\Z$ such that there is an affine open subset $W$ of $X$ containing $\beta(U)$. Let $S$ be a multiplicative subset of $\cO_{X_\Z}(W_\Z)$ such that $S^{-1}\cO_{X_\Z}(W_\Z)\simeq\cO_{X_\Z}(U)$. Then
 $$ \lim_{\substack{\longrightarrow \\ V\subset X\text{ open}\\ \text{s.t.\ }\beta(U)\subset V}}(\,\cM(V)\otimes_{\cO_X(V)}\cO_{X_\Z}(U)\,) \quad = \quad S^{-1}(\cM(W)_\Z). $$
\end{lemma}

\begin{proof}
 Note that $W$ occurs in the system of all open subsets $V$ of $X$ that contain $\beta(U)$. We call this system $\cD$ for short.

 Let $V_1,V_2\in\cD$ such that $V_2\subset V_1$ and let $\res_{V_1,V_2}:\cM(V_1)\to\cM(V_2)$ be the restriction map. By Corollary \ref{corollary_char_q-coh}, $S_{V_1,V_2}^{-1}\cM(V_1)\simeq\cM(V_2)$ when $S_{V_1,V_2}$ is a multiplicative subset of $\cO_X(V_1)$ such that $S_{V_1,V_2}^{-1}\cO_X(V_1)\simeq\cO_X(V_2)$. By Lemma \ref{localization_of_A_and_M}, $S_{V_1,V_2}^{-1}\cM(V_1)\simeq\cM(V_1)\otimes_{\cO_X(V_1)}\cO_X(V_2)$, and thus, in turn,
 \begin{equation*}\begin{split} \cM(V_2)\otimes_{\cO_X(V_2)}\cO_{X_\Z}(U) & \quad \simeq \quad S_{V_1,V_2}^{-1}\cM(V_1)\otimes_{\cO_X(V_2)}\cO_{X_\Z}(U) \\ & \quad \simeq \quad \cM(V_1)\otimes_{\cO_X(V_1)}\cO_X(V_2)\otimes_{\cO_X(V_2)}\cO_{X_\Z}(U) \\ & \quad \simeq \quad \cM(V_1)\otimes_{\cO_X(V_1)}\cO_{X_\Z}(U). \end{split}\end{equation*}

 Note for all $V_1,V_2\in\cD$, also $V_1\cap V_2\in\cD$. The colimit of the lemma is taken over a system of isomorphisms, and is isomorphic to $\cM(V)\otimes_{\cO_X(V)}\cO_{X_\Z}(U)$ for every $V\in\cD$. In particular the case when $V=W$ the colimit is $\cM(W)\otimes_{\cO_X(W)}\cO_{X_\Z}(U)$.

 To finish the proof, we calculate
 \begin{equation*}\begin{split} \cM(W)\otimes_{\cO_X(W)}\cO_{X_\Z}(U) & \quad \simeq \quad \cM(W)\otimes_{\cO_X(W)}\cO_{X}(W)_\Z\otimes_{\cO_{X}(W)_\Z}\cO_{X_\Z}(U) \\ & \quad \simeq \quad \cM(W)_\Z\otimes_{\cO_{X}(W)_\Z}\cO_{X_\Z}(U) \\ & \quad\simeq\quad S^{-1}\cM(W)_\Z. \hspace{5,4cm}\qedhere  \end{split}\end{equation*}
\end{proof}

Let $\QCoh X_\Z$ and $\Coh X_\Z$ denote the categories of quasi-coherent resp.\ coherent sheaves on $X_\Z$.

\begin{theorem}
 The base extension functor to $\Z$ restricts to a functor
 $$ (-)\otimes_{\cO_X}\cO_{X_\Z}:\QCoh X \to \QCoh X_\Z. $$
 If $X$ is Noetherian, then the base extension functor to $\Z$ restricts to a functor
 $$(-)\otimes_{\cO_X}\cO_{X_\Z}:\Coh X \to \Coh X_\Z.$$
 In particular, if $\cM$ is a quasi-coherent sheaf, $\{U_i\}$ is an affine cover of $X$ and $M_i=\cM(U_i)$, then the quasi-coherent sheaf $\cM_\Z$ is defined by the $\cO_{X_\Z}(U_i)$-modules $\cM_\Z(U_{i,\Z})=M_{i,\Z}$.
\end{theorem}

\begin{proof}
 Since in general for $\cM\in\cO_X-\Mod$, we have that $\cM_\Z(U_\Z)=\cM(U)_\Z$ for an open subset $U$ of $X$, the last statement follows from the first statement. Similarly, the second statement follows from the first statement by Theorem \ref{coherent_on_Noetherian_scheme} and the Lemma \ref{lemma:MIsFGIffMZIsFG}. Thus we are left with the first statement.

 Let $\cM$ be a quasi-coherent sheaf on $X$. Let $\cB$ be the collection of all affine open subsets of $X$. Then $\{V_\Z\}_{V\in\cB}$ covers $X_\Z$ and thus it suffices by Corollary \ref{corollary_char_q-coh} to show that for any open subset $U$ of any $W_\Z$ where $W\in\cB$, the restriction map $\res_{W_\Z,U}:\cM_\Z(W_\Z)\to\cM_\Z(U)$ extends to an isomorphism $S^{-1}\cM_\Z(W_\Z)\to\cM_\Z(U)$ where $S$ is a multiplicative subset of $\cO_{X_\Z}(W_\Z)$ such that $S^{-1}\cO_{X_\Z}(W_\Z)=\cO_{X_\Z}(U)$.

 Let $\cB(U)$ be the system of all $V\in\cB$ such that $\beta(U)\subset V$ together with the inclusion maps. Recall that $\cM_\Z(U)$ is defined as $\colim\cM(V)\otimes_{\colim\cO_X(V)}\cO_{X_\Z}(U)$ where the colimit is taken over $\cB(U)$. The universal properties of the tensor product and the colimit define mutually inverse isomorphisms
 $$ \xymatrix{\colim\cM(V)\otimes_{\colim\cO_X(V)}\cO_{X_\Z}(U)\ \ar[r]<3pt> & \ \colim \bigl(\, \cM(V)\otimes_{\cO_X(V)}\cO_{X_\Z}(U)\,\bigr) \ar[l]<3pt>}  $$
 that are induced by the identity maps $\cM(V)\to\cM(V)$ and $\cO_{X_\Z}(U)\to \cO_{X_\Z}(U)$. By Lemma \ref{lemma_for_base_ext_of_q-coh_sheaves}, the right colimit is isomorphic to $S^{-1}\cM(W)_\Z$, which is all we have to prove.
\end{proof}

\subsubsection{Locally  projective $\cO_X$-modules}

An $\cO_X$-module $\cM$ is said to be \emph{locally free (of rank $r$)} if it is quasi-coherent and if there is an affine cover $\{U_i\}$ of $X$ such that $\cM(U_i)$ is a free $\cO_X(U_i)$-set (of rank $r$). If $X=\Spec(R)$ is a usual affine scheme of a ring $R$, then locally free sheaves over $X$ correspond to projective $R$-modules. However, if $X=\Spec(A)$ is an affine $\Mo$-scheme, locally free sheaves over $X$ correspond to free $A$-sets because of the existence of the unique maximal ideal of $A$. This fact leads us to the following definition.

\begin{definition}\label{definition:LocallyProjectiveSheaves}
 An $\cO_X$-module $\cM$ is said to be \emph{locally projective} if it is quasi-coherent and if there is an affine cover $\{U_i\}$ of $X$ such that $\cM(U_i)$ is a projective $\cO_X(U_i)$-set.
\end{definition}

We let $\Proj X$ denote the full subcategory of $\Coh X$ which contains only the locally projective sheaves. So an $\cO_X$-module $\cM$ is in $\Proj X$ if and only if $\cM$ is locally projective and $\cM$ is coherent. By Corollary \ref{coro:LocalizationAndBaseExtensionOfProjectivesAreProjectives}, the base extension of a locally projective $\cO_X$-module is a locally free $\cO_{X_\Z}$-module. Thus we obtain a map $(-)\otimes_{\cO_X}\cO_{X_\Z}: \Proj X  \to \Bun (X_\Z)$, which is functorial in $X$.

\subsubsection{Examples}
\label{coherent_examples}

Let $A$ be a monoid and $X=\spec A$. By Corollary \ref{Cor_equiv_A-Mod_and_Coh_X}, the category of finitely generated $A$-modules is equivalent to $\Coh X$. We described these categories already in section \ref{A-set_examples}.

As an example that does not come from affine schemes, we describe the locally free sheaves on $X=\P^1_\Fun$. Recall from Section \ref{Examples_of_Mo-schemes} that $\P^1_\Fun$ is covered by two affine subschemes $U_1=\spec \Fun[T_1]$ and $U_2=\spec \Fun[T_2]$, which are both isomorphic to $\A^1_\Fun$ and that intersect in an affine open subscheme $U=\spec \Fun[T,T^{-1}]$, which is isomorphic to $\G_{m,\Fun}$. 
Let $S_i=\Fun[T_i]^\int=\Fun[T_i]-\{0\}$ (for $i=1,2$). A locally free sheaf on $\P^1_\Fun$ corresponds to a pair $(M_1,M_2)$ where $M_i$ is a finitely generated $\Fun[T_i]$-module (for i=$1,2$) together with an isomorphism $S_1^{-1}M_1\to S_2^{-1}M_2$. This means that $M_1$ and $M_2$ are necessarily of the same rank $r$. Let $m_1,\dotsc,m_r\in M_1$ such that $M_1=\cup_{i=1,\dotsc,r}\Fun[T_1].m_i$. If the rank is $r=1$, then the choice of an isomorphism corresponds to the choice of image of $m_1$ in $(S_2^{-1}M_2-\{0\})\simeq (\Fun[T_2^{\pm1}]-\{0\})$. Thus every $l\in \Z$ defines an isomorphism class of locally free sheaves of rank $1$ on $\P^1_\Fun$ by mapping $m_1$ to $T_2^l$.

The tensor product of two locally free sheaves of rank $1$ is again a locally free sheaf of rank $1$. This endows the set of isomorphism classes of locally free sheaves of rank $1$ with a group structure. We call this group $\Pic X$, the \emph{Picard group of $X$}, and $\Pic\P^1_\Fun\simeq \Z$ as a group. This corresponds to the situation of usual scheme theory and, indeed, the base extension to $\Z$ yields a group isomorphism $\Pic\P^1_\Fun\stackrel\sim\longrightarrow\Pic \P^1_\Z$.

In the case of higher rank $r$, the choice of an isomorphism $S_1^{-1}M_1\to S_2^{-1}M_2$ corresponds to the choice of elements $n_1,\dotsc,n_r\in S_2^{-1}M_2-\{0\}$ such that their orbits are pairwise different. Thus we obtain a bijection between the orbits of $S_1^{-1}M_1$ and the orbits of $S_2^{-1}M_2$. This shows that every rank $r$ bundle is the wedge product of coherent sheaves of rank $1$. This is the analogue to the result of Grothendieck in usual scheme theory that states that every vector bundle over the projective line decomposes into a direct sum of line bundles. Indeed, the base extension to $\Z$ yields a bijection $\Bun\P^1_\Fun\to\Bun\P^1_\Z$ between the isomorphism classes of locally free sheaves of rank $r$ on $\P_\Fun^1$ and the isomorphism classes of vector bundles of rank $r$ on $\P_\Z^1$.


\section{$\Fun$-schemes and sheaves}\label{SheavesOfModulesOnF1Schemes}

In this section, we review the notion of $\Fun$-schemes as introduced by Connes and Consani in \cite{CC09} and give an overview of examples with reference to \cite{LL09a}. We will proceed with establishing sheaves on $\Fun$-schemes. After proving some general results, we define admissible short exact sequences of locally free sheaves, which will be needed for the definition of $G$-theory and $K$-theory in the next section.

\subsection{$\Fun$-schemes}

After defining $\Fun$-schemes and the properties of $\Fun$-schemes that we will need in the rest of the paper, we will describe a selection of examples that shall give a flavour of what an $\Fun$-scheme is.

\subsubsection{Definition}

Recall from \cite{CC09} that an \emph{$\Fun$-scheme} is a triple $\cX=(\msch X, X, e_X)$ where $\msch X$ is an $\Mo$-scheme, $X$ is a scheme and $e_X: \msch X_\Z\to X$ is a morphism of schemes such that $e_X(k):\msch X_\Z(k)\to X(k)$ is a bijection of sets for every field $k$. We call $e_X$ the \emph{evaluation map}.

Let $\cX=(\msch X, X, e_X)$ and $\cY=(\msch Y,Y,e_Y)$ be $\Fun$-schemes. A \emph{morphism of $\Fun$-schemes $\cX\to\cY$} is a pair $\Phi=(\msch\varphi,\varphi)$ where $\msch\varphi:\msch X\to\msch Y$ is a morphism of $\Mo$-schemes and $\varphi:X\to Y$ is a morphism of schemes such that the diagram
$$\xymatrix{\msch X_\Z\ar[rr]^{\msch\varphi_\Z}\ar[d]_{e_X}&& \msch Y_\Z\ar[d]^{e_Y}\\ X\ar[rr]^{\varphi} && Y} $$
commutes.

The \emph{base extension functor $-\otimes_\Fun\Z$} associates to an $\Fun$-scheme $\cX=(\msch X, X,e_X)$ the scheme $X$ and to a morphism $\Phi=(\tilde\varphi,\varphi)$ of $\Fun$-schemes the morphism $\varphi$ of schemes.

The following is proven in the same way as \cite[Lemma 1.3 (1)]{LL09a}.

\begin{lemma}\label{e_x_is_a_continuous_bijection}
 Let $\cX=(\msch X,X,e_X)$ be an $\Fun$-scheme. Viewed as a map between the underlying topological spaces, $e_X:\msch X_\Z\to X$ is a continuous bijection.
\end{lemma}

We call an $\Fun$-scheme $\cX=(\msch X, X, e_X)$ \emph{integral} (\emph{Noetherian} / \emph{locally of finite type} / \emph{of finite type}) if both $\msch X$ and $X$ are integral (Noetherian / locally of finite type / of finite type). The $\Fun$-scheme $\cX$ is \emph{separable} if both $\msch X_\Z$ and $X$ are separable schemes.

\subsubsection{Examples}

 Every $\Mo$-scheme $\msch X$ has an associated $\Fun$-scheme $(\msch X,\msch X_\Z, \id)$. A morphism of $\Mo$-schemes $\msch\varphi:\msch X\to\msch Y$ defines the morphism $(\msch\varphi,\msch\varphi_\Z)$ of the associated $\Fun$-schemes, and, vice versa, every morphism $(\msch\varphi,\varphi)$ between the associated $\Fun$-schemes of $\msch X$ and $\msch Y$ is of this form since the evaluation maps are isomorphisms. Thus we obtain a fully faithful embedding of the category of $\Mo$-schemes into the category of $\Fun$-schemes. Its essential image are those $\Fun$-schemes whose evaluation map is an isomorphism.

 From now on, we will identify the category of $\Mo$-schemes with the essential image of this embedding and use the term $\Mo$-scheme for $\Fun$-schemes whose evaluation map is an isomorphism. In particular, if $A$ is a monoid, we will write $\Spec A$ for the $\Fun$-scheme $(\spec A, \Spec A_\Z, \id)$.

 A larger class of examples is delivered by the concept of a torified variety as introduced by L\'opez Pe\~na and the second author in \cite{LL09a}. A \emph{torified scheme} is a scheme together with a \emph{torification}, i.e.\ a morphism $e_T: T\to X$ from a scheme $T=\coprod_{i\in I}\Gm^{d_i}$ to $X$ (for a certain index set $I$ and $d_i\geq0$) such that for every $i\in I$, the restriction $\Gm^{d_i}\to X$ is a (locally closed) immersion and such that for every field $k$, the map $e_T(k):T(k)\to X(k)$ is a bijection. Examples of schemes that admit a torification are toric varieties, Schubert varieties and split reductive group schemes (cf. \cite[section 1.3]{LL09a}).

The scheme $T=\coprod_{i\in I}\Gm^{d_i}$ is the base extension of the $\Mo$-scheme $\msch T=\coprod_{i\in I}\G_{m,\Fun}^{d_i}$ to $\Z$. Thus a torification $e_T:T\to X$ yields an $\Fun$-scheme $\cT=(\msch T, X, e_T)$. The base extension of $\cT$ to $\Z$ is $X$, thus we obtain models of toric varieties, Schubert varieties and split reductive group schemes over $\Fun$. Note that $\msch T$ is affine and that its topological space is discrete and note that $\cT=(\msch T, X, e_T)$ is of finite type if the index set $I$ is finite.


\subsection{Sheaves on $\Fun$-schemes}

Let $\cX=(\msch X,X,e_X)$ be an $\Fun$-scheme. A \emph{sheaf} on $\cX$ is a triple $\cM=(\msch M,M,\epsilon_M)$ where $\msch M$ is a sheaf on $\msch X$, $M$ is a sheaf on $X$ and $\epsilon_M$ is an isomorphism $\epsilon_M: \msch M_\Z\stackrel\sim\longrightarrow e_X^\ast M$ of sheaves on $\msch X_\Z$.

Let $\cM=(\msch M,M,\epsilon_M)$ and $\cN=(\msch N,N,\epsilon_N)$ be sheaves on $\cX$. A \emph{morphism $\Psi:\cM\to\cN$ of sheaves on $\cX$} is a pair $\Psi=(\msch\psi,\psi)$ where $\msch\psi: \msch M\to\msch N$ is a morphism of sheaves on $\msch X$ and $\psi:M\to N$ is a morphism of sheaves on $X$ such that the diagram
$$\xymatrix{\msch M_\Z\ar[d]_{\epsilon_M}\ar[rr]^{\msch\psi_\Z} && \msch N_\Z\ar[d]^{\epsilon_N} \\ e_X^\ast M\ar[rr]^{e_X^\ast(\psi)}&& e_X^\ast N} $$
commutes.

We call a sheaf $\cM=(\msch M,M,\epsilon_M)$ an \emph{$\cO_\cX$-module} if $\msch M$ is an $\cO_{\msch X}$-module and $M$ is an $\cO_X$-module. A morphism of $\cO_\cX$-modules is a morphism of sheaves that respects the $\cO_{\msch X}$-module and the $\cO_X$-module structure. We denote the category of $\cO_\cX$-modules by $\cO_\cX-\Mod$. We call $\cM$ a \emph{(quasi-) coherent sheaf on $\cX$} if both $\msch M$ and $M$ are (quasi-)coherent, and we denote the full subcategories of $\cO_\cX-\Mod$ whose objects are quasi-coherent (resp.\ coherent) sheaves on $\cX$ by $\QCoh\cX$ (resp.\ $\Coh \cX$). We denote by $\Bun\cX$ the full subcategory of $\Coh\cX$ of \emph{locally free sheaves}, i.e.\ sheaves $\cM=(\msch M,M,\epsilon_M)$ where  $\msch M$  is locally projective and $M$ is locally free.

Let $\cX=(\msch X,X,e_X)$ be an $\Mo$-scheme, i.e.\ $e_X:\msch X_\Z\to X$ is an isomorphism, and let $\msch M$ be a sheaf on $\msch X$. Then $M=(e_X)_\ast \msch M_\Z$ is a sheaf on $X$ and there is a canonical isomorphism $\epsilon:\msch M_\Z\to e_X^\ast M$. The triple $\cM=(\msch M,M,\epsilon_M)$ is thus a sheaf on $\cX$. Consequently, every morphism $\msch\psi:\msch M\to\msch N$ of sheaves on $\msch X$ extends uniquely to a morphism $(\msch\psi,\psi):\cM\to\cN$ between the associated sheaves $\cM$ and $\cN$ of $\msch M$ and $\msch N$, respectively, and, vice versa, every morphism $(\msch\psi,\psi):\cM\to\cN$ is of this sort. This describes an equivalence of categories between the category of sheaves on $\msch X$ and the category of sheaves on $\cX$. Thus we obtain:

\begin{proposition}\label{prop:ModulesOverMoSchemes}
 If $\cX=(\msch X, \msch X_\Z, e_X)$ is an $\Mo$-scheme, then the association $\msch M \mapsto (\msch M,M,\epsilon_M)$ with $M=(e_X)_\ast\msch M_\Z$ and $\epsilon:\msch M_\Z\to e_X^\ast M$ the canonical isomorphism defines an equivalence of categories
 $$ \cO_{\msch X}-\Mod \stackrel\sim\longrightarrow \cO_\cX-\Mod $$
 that restricts to equivalences of the following subcategories:
 $$ \QCoh \msch X \stackrel\sim\longrightarrow \QCoh\cX, \qquad  \Coh \msch X \stackrel\sim\longrightarrow \Coh\cX, \qquad \Proj \msch X \stackrel\sim\longrightarrow \Bun\cX.\qed $$
\end{proposition}

Let $\cX=(\msch X,X,e_X)$ be an arbitrary $\Fun$-scheme. Let $\msch 0$ be the zero sheaf on $\msch X$, let $0$ be the zero sheaf on $X$ and let $0:\msch0_\Z\to e_X^\ast 0$ be the zero morphism. Then $\0=(\msch 0,0,0)$ is the zero object in the category of $\cO_\cX$-modules, and it is contained in the subcategories of (quasi-)coherent and locally free sheaves. Given two $\cO_X$-modules $\cM$ and $\cN$, then there is a unique morphism $0:\cM\to\cN$ that factors through $\0$. This turns $\Hom(\cM,\cN)$ into a pointed set and gives it the structure of an $\Fun$-set.

In particular, one can define the kernel and cokernel of a morphism $\varphi=(\msch f,f):\cM\to\cN$ (if it exists) as the equalizer and coequalizer of $\varphi$ and $0$. 

\begin{corollary}
 If $\cX$ is an $\Mo$-scheme, then every morphism of $\cO_\cX$-modules has both a kernel and a cokernel.
\end{corollary}

\begin{proof}
 Following Proposition \ref{prop:ModulesOverMoSchemes}, the category of $\cO_{\cX}$-modules is equivalent to the category of $\cO_{\tilde X}$-modules. Thus we reduced the question of the existence of kernels and cokernels to $\cM_0$-schemes, which can be easily solved by means of Proposition 2.13.
\end{proof}

Let $\cX=(\msch X,X,e_X)$ be an arbitrary $\Fun$-scheme. Let $\varphi=(\msch f,f):\cM\to\cN$ be a morphism between $\cO_X$-modules $\cM=(\msch M,M,\epsilon_M)$ and $\cN=(\msch N,N,\epsilon_N)$. Since taking the direct image of a sheaf commutes with base extension (cf.\ Lemma \ref{base_extension_properties_for_sheaves}), the triple $\im\varphi:=(\im\msch f,\im f,\epsilon_N\vert_{\im\msch f_Z})$ is a sheaf, and it is the image of $\varphi$. This proves

\begin{lemma}\label{lemma:CokernelExists}
 Let $\cX$ be an arbitrary $\Fun$-scheme and $\varphi$ be a morphism of $\cO_\cX$-modules. Then the cokernel of $\varphi$ exists.\qed
\end{lemma}

One checks easily that the coproduct of two $\cO_\cX$-modules $\cM=(\msch M,M,\epsilon_M)$ and $\cN=(\msch N,N,\epsilon_N)$ is $\cM\vee\cN=(\msch M\vee\msch N,M\oplus N, \epsilon_M\oplus\epsilon_N)$.

\subsubsection{Normal morphisms}

We extend the notion of normal morphisms to $\Fun$-schemes. If we write scheme in the following, we always mean a scheme in the usual sense.

\begin{definition}
 Let $\varphi: X\to Y$ be a morphism of schemes and $f: M\to N$ be a morphism of $\cO_Y$-modules. Then $f$ is called \emph{$\varphi$-flat} if
 $$0\longrightarrow \varphi^*(\ker(f))\longrightarrow \varphi^*M\longrightarrow \varphi^* N$$
 is an exact sequence of $\cO_X$-modules.
\end{definition}

In other words, $f: M\to N$ is $\varphi$-flat if and only if the canonical morphism $\varphi^*\ker(f)\to \ker(\varphi^*f)$ is an isomorphism. In particular, a monomorphism $f: M\to N$ is $\varphi$-flat if and only if $\ker(\varphi^*f)=\0$.

The notion of a \emph{$\varphi$-flat morphism} between modules over a ring $S$ and a ring homomorphism $\varphi: R\to S$ is defined analogously.

\begin{definition}
 Let $\cX=(\msch X,X,e_X)$ be an $\Fun$-scheme. A morphism of $\cO_\cX$-modules
 $$(\msch f,f):\cM=(\msch M,M,\epsilon_M)\longrightarrow \cN=(\msch N,N,\epsilon_N)$$
 is called \emph{normal} if the morphism $\msch f:\msch M\to \msch N$ of $\cO_{\msch X}$-modules is normal and $f:M \to N$ is $e_X$-flat.
\end{definition}

Since every morphism $f:M\to N$ is $\varphi$-flat if $\varphi$ is an isomorphism and since $\cO_{\msch X}-\Mod \stackrel\sim\longrightarrow \cO_\cX-\Mod$ if $\cX=(\msch X,X, e_X)$ is an $\Mo$-scheme, we have the following fact.

\begin{lemma}
 Let $\cX=(\msch X,\msch X_\Z, e_X)$ be an $\Mo$-scheme and let $(\msch f,f):\cM\to\cN$ be a morphism of $\cO_\cX$-modules. Then $(\msch f,f)$ is normal if and only $\msch f$ is normal.\qed
\end{lemma}

The above lemma implies that the notion of normal morphisms of sheaves of modules over general $\Fun$-schemes is consistent with the notion of normal morphisms of sheaves of modules over $\Mo$-schemes. In particular, normal morphisms of sheaves of modules over $\Mo$-schemes are closed under compositions by Proposition \ref{prop:NmMrphsmAreClosedUnderComposition}. Over general $\Fun$-schemes, it is not clear to the authors whether normal morphisms are closed under composition or not.

\begin{lemma}\label{lemma_normal_morphisms_over_f1-schemes}
If $(\msch f,f):(\msch M,M,\epsilon_M) \to (\msch N,N,\epsilon_N)$  is a normal morphism of quasi-coherent sheaves, then $\ker(\msch f, f)$ exists.
\end{lemma}
\begin{proof}
 We can extend the commutative diagram in the definition of a morphism of sheaves to the following diagram:
 $$\xymatrix{  \0 \ar[r]\ar[d]^{\cong} & (\ker\msch f)_\Z \ar[r] & \msch M_\Z \ar[r]^{\msch f_\Z}\ar[d]^{\cong}_{\epsilon_M} &\msch N_\Z\ar[d]^{\cong}_{\epsilon_N}    \\
        \0 \ar[r]              & e_X^\ast\ker f \ar[r]    & e_X^\ast M \ar[r]^{e_X^\ast f}                       &e_X^\ast N.} $$
 Since $\msch f$ is normal, the top row is exact by Lemma \ref{extension_of_normal_morphism_to_Z} and the local characterization of quasi-coherent sheaves. The bottom row is exact because $f$ is $e_X$-flat. It follows from the diagram that $$\epsilon_M\vert_{\ker\msch f_\Z}: (\ker\msch f)_\Z  \stackrel{\cong}\longrightarrow  e_X^\ast\ker f$$ is an isomorphism. So the triple $(\ker(\msch f), \ker(f), \epsilon_M\vert_{\ker\msch f_\Z})$ is an object in $\cO_\cX-\Mod$ and it is a subobject of  $(\msch M,M,\epsilon_M)$. It is routine to check that $(\ker(\msch f), \ker(f), \epsilon_M\vert_{\ker\msch f_\Z})$ is the kernel of $(\msch f, f).$
\end{proof}

\subsubsection{Admissible short exact sequences}

A sequence
$$ \xymatrix{\cM_1\ar[r]^\varphi & \cM_2\ar[r]^\psi & \cM_3 } $$
of $\cO_\cX$-modules is said to be \emph{exact at $\cM_2$} if the kernel of $\psi$ exists and $\im\varphi=\ker\psi$. A \emph{short exact sequence of $\cO_\cX$-modules} is a five term sequence
$$ \xymatrix{\0\ar[r] & \cM_1\ar[r]^\varphi & \cM_2\ar[r]^\psi & \cM_3 \ar[r] & \0} $$
that is exact at $\cM_1$, $\cM_2$ and $\cM_3$. We call this short exact sequence \emph{admissible} if both $\varphi$ and $\psi$ are normal morphisms. A morphism is called an admissible monomorphism or admissible epimorphism  if it appears as $\varphi$ resp.\ $\psi$ in some admissible short exact sequence. Admissible short exact sequences can be characterized more explicitly as follows.

\begin{proposition} \label{prop:DescriptionOfNormalMorphisms}
 Let $\cM_1=(\msch M_1,M_1,\epsilon_{M_1})$, $\cM_2=(\msch M_2,M_2,\epsilon_{M_2})$ and $\cM_3=(\msch M_3,M_3,\epsilon_{M_3})$ be sheaves of modules over the $\Fun$-scheme $\cX=(\msch X, X, e_X)$. A sequence $$\xymatrix{\0\ar[r] & \cM_1\ar[r]^{(\msch f,f)} & \cM_2\ar[r]^{(\msch g,g)} & \cM_3 \ar[r] & \0}$$ is an admissible short exact sequence if and only if the following statements are true.
 \begin{enumerate}
  \item The sequence $\xymatrix{\msch 0 \ar[r] & \msch M_1\ar[r]^{\msch f} & \msch M_2\ar[r]^{\msch g} & \msch M_3 \ar[r] & \msch 0}$ is an admissible short exact sequence of $\cO_{\msch X}$-modules.
  \item The sequence $\xymatrix{0 \ar[r] & M_1\ar[r]^{f} & M_2\ar[r]^{g} &  M_3 \ar[r] & 0}$ is a short exact sequence of $\cO_{X}$-modules and $f$ and $g$ are $e_X$-flat.
 \end{enumerate}
\end{proposition}

Tensor product, pull back and push forward of sheaves of modules over $\Fun$-schemes can also be defined in a natural manner. We postpone the discussion of these constructions to the next chapter when we need these constructions to study K-theory.


\section{$G$-theory and $K$-theory of $\Fun$-schemes}\label{subsection:definition_of_k-theory}

In this section we introduce $G$-theory and $K$-theory for $\Fun$-schemes. We begin with the more involved case of $G$-theory in a first subsection and conclude from this that $K$-theory exists in a second subsection.

\subsection{$G$-theory of Noetherian $\Fun$-schemes }
Recall that an $\Fun$-scheme $\cX=(\msch X, X, e_X)$ is Noetherian if and only if both $\msch X$ and $X$ are Noetherian. Also recall that we denote the category of coherent sheaves over $\cX$ by $\Coh\cX$. In this subsection we prove that the collection of admissible short exact sequences in $\Coh\cX$ makes $\Coh\cX$ a quasi-exact category in the sense of \cite{Deitmar06}. As for usual schemes, we require $\cX$ to be Noetherian because we need $\Coh\cX$ to be closed under taking quotients and subobjects.

\subsubsection{Definition}

In order to define $G$-theory we need to prove that for an $\Fun$-scheme $\cX$, the category $\Coh\cX$ together with the class of admissible exact sequences is an exact category, which will be done in Proposition \ref{prop:QuasiExactStructureOnCoh(X)}.

\begin{lemma}\label{lemma:SomeE-FlatMorphismsAreClosed} Let $e:R\to S$ be  a ring homomorphism. Let $f: M\to N$ and $g: N\to K$ be morphisms of $R$-modules.
\begin{enumerate}\item If $f$ and $g$ are injective and $e$-flat, then $g\circ f$ is also injective and $e$-flat.
                 \item If $f$ and $g$ are surjective and $e$-flat, then $g\circ f$ is also surjective and $e$-flat.
                 \item If $f$ is injective and $e$-flat, then the quotient map $N\to N/f(M)$ is $e$-flat; if $g$ is surjective and $e$-flat, then the inclusion map $\ker(g)\to N$ is $e$-flat.
\end{enumerate}
\end{lemma}

\begin{proof}
 For the first statement, by the assumption we see that $e^*f$ and $e^*g$ are both injective. So $e^*(g\circ f)$ is also injective which implies that $g\circ f$ is $e$-flat.

 For the second statement, it is easy to see that $g\circ f$ is surjective.  We consider the following commutative diagram of $S$-modules.
 $$\xymatrix{                           & \ker(f)\otimes S \ar[r] \ar[d]^\cong & \ker(g\circ f) \otimes S \ar[r] \ar[d]  & \ker(g) \otimes S \ar[r] \ar[d]^\cong  & 0\\
                        0 \ar[r]             & \ker(e^*f) \ar[r]             & \ker(e^*(g\circ f)) \ar[r]& \ker(e^*g) \ar[r] & 0.                                               }$$
 In the above diagram,  the vertical arrows are all the natural maps and the top row comes from the short exact sequence of $R$-modules
 $$0\longrightarrow \ker(f) \longrightarrow \ker(g\circ f) \longrightarrow \ker(g) \longrightarrow 0.$$
 Since the left and right vertical arrows are both isomorphisms by assumption, the middle arrow is also an isomorphism. So $g\circ f$ is also $e$-flat.

 The proof of the last statement is trivial.
\end{proof}

\begin{lemma}\label{lemma:PullBackOfEFlatMaps} Let $e:R\to S$ be  a ring homomorphism. Let $f: M\to N$ and $g: P\to N$ be $e$-flat morphisms of $R$-modules where $f$ is injective while $g$ is surjective. Let $g^{-1}(f(M))=M\times_NP$ be the fiber product and let $f'$ and $g'$ be the pull back of $f$ and $g$ respectively. Then both $f'$ and $g'$ are $e$-flat. In particular, $(M\times_NP)\otimes S =(M\otimes S)\times_{(N\otimes S)} (P\otimes S)$.
\end{lemma}
\begin{proof}Notice that  $f'$ is injective and $g'$ is surjective. Also notice that $\coker(f)=\coker(f')$ and $\ker(g)=\ker(g')$. The proof is similar to the proof of Lemma \ref{lemma:SomeE-FlatMorphismsAreClosed} by considering the following commutative diagram of $S$-modules. Details are skipped.
 $$\xymatrix{                          &                                                   &                                                              & 0\ar[d]  \\
                                                 & \ker(g)\otimes S \ar[r] \ar[d]   &(M\times_NP) \otimes S \ar[r] \ar[d]^{f'\otimes Id}  & M \otimes S \ar[r] \ar[d]^{f\otimes Id} &  0\\
                        0 \ar[r]             & \ker(e^*g) \ar[r]             & P \otimes S \ar[r]^{g\otimes Id} & N\otimes S \ar[r] & 0.
}$$
 \end{proof}

The next proposition implies that the collection of admissible exact sequences makes $\Coh\cX$ a {\it quasi-exact category} in the sense of \cite{Deitmar06}.
\begin{proposition}\label{prop:QuasiExactStructureOnCoh(X)} Let $\cX=(\msch X, X, e_X)$ be a Noetherian $\Fun$-scheme. The category  $\Coh\cX$ is a quasi-exact category.  That is, the following statements are true in  $\Coh\cX$.
\begin{enumerate}
 \item A five term sequence that is isomorphic to an admissible short exact sequence is an admissible short exact sequence.
 \item Admissible monomorphisms are closed under composition; so are admissible epimorphisms.
 \item Given an admissible monomorphism $(\msch i, i): \cM\to \cN$ and an admissible epimorphism $(\msch j, j): \cP \to \cN$, the pull back exists in $\Coh\cX$ and the pull back of $i$ is an admissible monomorphism and the pull back of $j$ is an admissible epimorphism.
\end{enumerate}
\end{proposition}
\begin{proof} The first statement is obvious.

For the second statement, assume that we have admissible short exact sequences
$$\xymatrix{\0\ar[r] & \cM\ar[r]^{(\msch \alpha,\alpha)} & \cN\ar[r]^{(\msch \beta,\beta)} & \cK \ar[r] & \0 \quad \mbox{ and } \quad
            \0\ar[r] & \cN\ar[r]^{(\msch f,f)} & \cS\ar[r]  & \cT \ar[r] & \0
}$$
in $\Coh\cX$ where $\cM=(\msch M, M, \epsilon_M), \cN=(\msch N, N, \epsilon_N)$ and $\cS=(\msch S, S, \epsilon_S).$ By Lemma \ref{lemma:CokernelExists}, the cokernel of $(\msch f, f)\circ (\msch \alpha, \alpha)$ exists and equals $$(\coker(\msch f \circ \msch \alpha), \coker(f \circ \alpha), \bar{\epsilon}_S).$$ Let $(\msch \pi, \pi)$ be the natural projection $\cS \to \coker((\msch f, f)\circ (\msch \alpha, \alpha) ).$ It is obvious that the following sequence is an admissible short exact sequence in $\cO_{\msch X}$.
$$\xymatrix{\msch 0\ar[r] & \msch M\ar[r]^{\msch f\circ \msch \alpha} & \msch S\ar[r]^-{\msch \pi} & \coker(\msch f \circ \msch \alpha) \ar[r] & \msch 0. }$$
By Lemma \ref{lemma:SomeE-FlatMorphismsAreClosed}, the morphisms in the following exact sequence of $\cO_X$-modules are $e_X$-flat.
$$\xymatrix{0\ar[r] & M\ar[r]^{f\circ \alpha} & S \ar[r]^-{\pi} & \coker( f \circ \alpha) \ar[r] & 0. }$$
So, by Proposition \ref{prop:DescriptionOfNormalMorphisms}, we have the following admissible short exact sequence which proves that $(\msch f, f)\circ (\msch \alpha, \alpha)$ is an admissible monomorphism.
$$\xymatrix{
\0\ar[r] & \cM\ar[rr]^{(\msch f, f)\circ(\msch \alpha, \alpha)} && \cS\ar[r]^-{(\msch \pi,\pi)} & \coker((\msch f,f)\circ(\msch \alpha, \alpha)) \ar[r] & \0.
}$$

Assume that there is an admissible short exact sequence in $\Coh\cX$
$$\xymatrix{ \0 \ar[r] & \cU \ar[r] & \cK\ar[r]^{(\msch g,g)} & \cV \ar[r] & \0. }$$ By Lemma \ref{lemma:SomeE-FlatMorphismsAreClosed}, we see that
$(\msch g,g)\circ (\msch \beta,\beta) $ is normal and hence it has a kernel, which is given by
$$(\ker(\msch g \circ \msch \beta), \ker(g \circ \beta), \epsilon_N).$$ Similar to the case of admissible monomorphisms, one checks that the following sequence is an admissible short exact sequence in $\cO_\cX$
$$\xymatrix{
\0\ar[r] & \ker((\msch g,g)\circ (\msch \beta,\beta)) \ar[r] & \cN\ar[rr]^-{(\msch g,g)\circ (\msch \beta,\beta)} && \cV \ar[r] & \0.
} $$
This completes the proof of the second statement.

Now we prove the last statement. The diagram $\cM\to \cN \leftarrow \cP$ gives two pull back diagrams
$$\xymatrix{ \msch K \ar[r]^{\msch{ i'}} \ar[d]_{\msch{j'}} & \msch P \ar[d]^{\msch{j}} & \ \ar@{}[d]|{{\text{\normalsize and}}}& K \ar[r]^{i'} \ar[d]_{j'} & P \ar[d]^{j}\\
                 \msch M \ar[r]^{\msch{i}}                   & \msch N                  & & M \ar[r]^i                   & N                                 }$$
in the categories $\Coh(\msch X)$ resp.\ $\Coh(X)$ where $\cP=(\msch P, P, \epsilon_P)$.
Applying the base extension functor to $\Z$ to the left Cartesian square and apply the pull back functor $e_X^*$ to the right Cartesian square, we get the diagrams
$$\xymatrix{ \msch K_\Z \ar[r]^{\msch{ i'}_\Z} \ar[d]_{\msch{j'}_\Z} & \msch P_\Z \ar[d]^{\msch{j}_\Z} & \ \ar@{}[d]|{{\text{\normalsize and}}} &  e_X^*K \ar[r]^{e_X^*i'} \ar[d]_{e_X^*j'} & e_X^*P \ar[d]^{e_X^*j}\\
                 \msch M_\Z \ar[r]^{\msch{i_\Z}}                   & \msch N_\Z  && e_X^*M \ar[r]^{e_X^*i}                   & e_X^*N
}$$
of $\cO_{\msch X_\Z}$-modules. One verifies immediately that the left diagram is a Cartesian diagram. The right diagram is Cartesian by Lemma \ref{lemma:PullBackOfEFlatMaps}.
Being the pull back of isomorphic diagrams, $\msch K_\Z$ and $ e_X^*K$ are naturally isomorphic. Let $\epsilon_K$ denote the isomorphism, then $(\msch K, K, \epsilon_K)$ is an object in $\Coh\cX$. It is routine to check that $\cK=(\msch K, K, \epsilon_K)$ makes the diagram
$$\xymatrix{ \cK \ar[r]^{(\msch i', i')} \ar[d]_{(\msch{j'}, j')} & \cP \ar[d]^{(\msch{j},j)} \\
                 \cM \ar[r]^{(\msch{i}, i)}                   & \cN
}$$
Cartesian in $\Coh\cX$ and that $(\msch i',i')$ and $(\msch j', j')$ are admissible monomorphism and epimorphism, respectively.
\end{proof}

By Proposition \ref{prop:QuasiExactStructureOnCoh(X)}, one can apply Quillen's Q-construction \cite{Quillen341} to the category $\Coh\cX$ as usual. Let $Q\Coh\cX$ be the category having the same objects as $\Coh\cX$. A morphism from $\cM$ to $\cN$ in $Q\Coh\cX$
is an isomorphism class of diagrams in $\Coh\cX$
$$\xymatrix{ \cM & \cP \ar@{->>}[l]-_j \ar@{>->}[r]^i & \cN,
}$$ where $i$ is an admissible monomorphism and $j$ is an admissible epimorphism. The above diagram is isomorphic to
$$\xymatrix{ \cM & \cP' \ar@{->>}[l]-_{j'} \ar@{>->}[r]^{i'} & \cN
}$$ if and only if there is an isomorphism $\tau: \cP \to \cP'$ such that $j=j'\circ \tau $ and $i= i'\circ \tau$. The composition can be defined in the same way as in \cite{Quillen341} because of Proposition \ref{prop:QuasiExactStructureOnCoh(X)}. Since the isomorphism classes of objects in $Q\Coh\cX$ form a set, the geometric realization of the nerve of $Q\Coh\cX$ is well defined, which is denoted as $|Q\Coh\cX|$.

\begin{definition} Let $\cX=(\msch X,X,e_X)$ be a Noetherian $\Fun$-scheme. The \emph{algebraic $G$-theory of $\cX$} is defined as
$$G_i(\cX)=\pi_{i+1} |Q\Coh\cX|, \quad \mbox{ where } i\geq 0.$$
\end{definition}

\subsection{$K$-theory of  $\Fun$-schemes}

Let $\cX$ be a general $\Fun$-scheme. In this subsection, we define algebraic $K$-theory of $\cX$ using the category $\Bun\cX$ of locally free $\cO_\cX$-modules of finite rank.

 \subsubsection{Definition}
 Let $\cX=(\msch X, X, e_X)$ be an $\Fun$-scheme. \emph{An admissible short exact sequence in $\Bun\cX$} is an admissible short exact sequence in $\Coh\cX$
\begin{equation}\label{ShortExactSequaceOfLocallyFreeSheavesOverF1Schemes}\0\longrightarrow (\msch{M}, M, \epsilon_M) \stackrel{(\msch i, i)}{\longrightarrow} (\msch{N}, N, \epsilon_N) \stackrel{(\msch j, j)}{\longrightarrow} (\msch{K}, K, \epsilon_K)\longrightarrow \0\end{equation}
where all objects are in $\Bun\cX$.
A morphism in $\Bun\cX$ is called a $K$-admissible monomorphism (resp.\ epimorphism) if it appears as the map $(\msch i, i)$ (resp.\ $(\msch j, j)$) in some exact sequence as (\ref{ShortExactSequaceOfLocallyFreeSheavesOverF1Schemes}).

\begin{proposition}\label{prop:DescriptionOfKAdmissibleMonoAndEpis}Let $\cX=(\msch X, X, e_X)$ be an $\Fun$-scheme. The sequence
$$\0\longrightarrow (\msch{M}, M, \epsilon_M) \stackrel{(\msch i, i)}{\longrightarrow} (\msch{N}, N, \epsilon_N) \stackrel{(\msch j, j)}{\longrightarrow} (\msch{K}, K, \epsilon_K)\longrightarrow \0$$ is an admissible short exact sequence in $\Bun\cX$ if and only if
$$\msch 0\longrightarrow \msch{M} \stackrel{\msch i}{\longrightarrow} \msch{N}\stackrel{\msch j}{\longrightarrow} \msch{K}\longrightarrow \msch 0 $$
is an admissible short exact sequence of locally projective $\cO_{\msch X}$-modules and
$$0\longrightarrow M \stackrel{i}{\longrightarrow} N \stackrel{j}{\longrightarrow} K \longrightarrow 0$$ is an exact sequence of locally free $\cO_X$-modules.
\end{proposition}
\begin{proof} Since $e_X^*$ is exact on locally free sheaves, the pull back along $e_X$ of the exact sequence $0\longrightarrow M \stackrel{i}{\longrightarrow} N \stackrel{j}{\longrightarrow} K \longrightarrow 0$ remains exact. This shows that $i$ and $j$ are automatically $e_X$-flat. So Proposition \ref{prop:DescriptionOfNormalMorphisms} applies here to give the result.
\end{proof}

\begin{lemma}\label{lemma:QuasiExactStructureOnBun(X)} Let $\cX=(\msch X, X, e_X)$ be an $\Fun$-scheme.
\begin{enumerate}
 \item A five term sequence that is isomorphic to an admissible short exact sequence in $\Bun\cX$ is an admissible short exact sequence in $\Bun\cX$.
 \item $K$-admissible monomorphisms are closed under compositions, so as $K$-admissible epimorphisms.
 \item Given a $K$-admissible monomorphism $(\msch i, i): \cM\to \cN$ and a $K$-admissible epimorphism $(\msch j, j): \cP \to \cN$, the pull back $\cM \times_{\cN} \cP$ is an object in $\Bun\cX$ and the pull back of $i$ is a $K$-admissible monomorphism and the pull back of $j$ is a $K$-admissible epimorphism.
\end{enumerate}
\end{lemma}
\begin{proof} The proof of Proposition \ref{prop:QuasiExactStructureOnCoh(X)} works here. Using the same notations as in the proof of Proposition \ref{prop:QuasiExactStructureOnCoh(X)}, we only need to check that the objects
$\coker((\msch f, f)\circ (\msch \alpha, \alpha) )$, $\ker((\msch g,g)\circ (\msch \beta,\beta))$ and $\cM \times_{\cN} \cP$ are objects in $\Bun\cX$. It is clear that $\coker( f\circ \alpha )$, $\ker(g\circ \beta)$ and $M \times_{N} P$ are locally free of finite rank over $X$. It remains to show that $\coker( \msch f \circ \msch\alpha )$, $\ker(\msch g\circ \msch\beta)$ and $\msch M \times_{\msch N} \msch P$ are locally projective and coherent on $\msch X$. This can be checked by assuming that $\msch X=\Spec(A)$ is affine. We prove in detail that $\msch M \times_{\msch N} \msch P$ is projective and a finitely generated $A$-set. The proof of the fact that  $\coker( \msch f \circ \msch\alpha )$ and $\ker(\msch g\circ \msch\beta)$ are projective and finitely generated is similar and hence skipped.

Since $\msch i: \msch M \to \msch N$ is admissible, $\coker(\msch i)$ is projective and there is an admissible short exact sequence
$$\msch 0\longrightarrow \msch M \longrightarrow \msch N \longrightarrow \coker(\msch i) \longrightarrow \msch0.$$ By Proposition \ref{prop:SplittingOfAdmissibleShortExactSequences}, we can assume that $\msch N=\msch M \vee \coker(\msch i)$ and $\msch i$ is the canonical inclusion. Similarly, we can assume that $P=\msch N \vee \ker(\msch j)$ and $\msch j$ is the canonical projection. This implies that $\ker(\msch j)$ is projective and finitely generated by Proposition \ref{prop:DescriptionOfProjectiveAsets}.  So $\msch M \times_{\msch N} \msch P = \msch M \vee \ker(\msch j)$. Since both $ \msch M$ and $\ker(\msch j)$ are projective and finitely generated, $\msch M \times_{\msch N} \msch P$ is also projective and finitely generated as desired.\end{proof}

Lemma \ref{lemma:QuasiExactStructureOnBun(X)} implies that the collection of admissible exact sequences in $\Bun\cX$ make $\Bun\cX$ a  quasi-exact category. In particular, one can apply Quillen's Q-construction to the category $\Bun\cX$ as usual. Let $Q\Bun\cX$ be the resulting category  and let $|Q\Bun\cX|$ denote the geometric realization of the nerve of $Q\Bun\cX$.

\begin{definition} Let $\cX$ be an $\Fun$-scheme. The \emph{algebraic $K$-groups of $\cX$} are defined as
$$K_i(\cX)=\pi_{i+1}(|Q\Bun\cX|), \quad \mbox{ where }i\geq 0.$$
\end{definition}

\subsection{The $K$-theory spectrum}

In this section, we employ Waldhausen's $S_{\bullet}$-construction to show that the $K$-theory of an $\Fun$-scheme $\cX$ is the infinite loop space $\Omega|S_{\bullet} (\Bun \cX)|$. In subsequent paragraphs, we investigate certain properties of the theory developed so far: if $H$ is an abelian group, then the $K$-theory of $\Spec\Fun[H]$ is isomorphic to the stable homotopy of $BH_+$ and $G_0(\Spec\Fun[H])$ is the Burnside ring of $H$. If $H$ is simple, then the $K$-theory of $\Spec\Fun[H]$ is a direct summand of its $G$-theory; exemplary, we calculate the $K$-theory of the affine line and of the spectrum of a monoid with a non-trivial idempotent. 

In contrast to the rest of the paper, in the following sections we assume a certain familiarity of the reader with the mentioned concepts. We therefore do not repeat all definitions and statements we use from topology, and  provide references as required.

A Waldhausen category is a category together with two distinct subclasses of morphisms, namely, cofibrations and weak equivalences, that are subject to certain axioms (see \cite[sections 1.1 and 1.2]{W85}). For a small Waldhausen category, one can apply Waldhausen's $S_{\bullet}$-construction to obtain an infinite loop space, and hence a spectrum, which is defined as the {$K$-theory space} or spectrum of the Waldhausen category, see  \cite[section 1.3]{W85}.

A Quillen exact category is naturally a Waldhausen category if we define cofibrations as admissible monomorphisms and define weak equivalences as isomorphisms (cf.\ \cite[section 1.9]{W85}). One can define a Waldhausen category structure on a quasi-exact category in the same way---we only need to verify that admissible monomorphisms in the quasi-exact category are stable under arbitrary cobase changes. This fact is checked in the following lemma. 

\begin{lemma} \label{cobasechange} Let $\Bun\cX$ be the category of locally projective sheaves over an $\Fun$-scheme $\cX=(\msch X, X, e_X)$. Let $i: \cM \to \cN$ be a $K$-admissible monomorphism and let $f: \cM \to \cP$ be an arbitrary morphism in $\Bun\cX$. Then the push-out $\cN\cup_{\cM}\cP$ exists in $\Bun\cX$ and the natural map $\cP \to \cN\cup_{\cM}\cP$ is $K$-admissible.
\end{lemma}

\begin{proof} First assume that $\cX=\spec(A)$ for some monoid $A$. Then $\cM=\msch M, \cN=\msch N, \cP=\msch P$ are associated to projective $A$-sets and, up to isomorphism, $i$ can be assumed to be the embedding $\msch M\subset\msch  N=\msch M\vee\msch  M'$ for some projective $A$-set $\msch M'$. So the push-out is $\msch M'\vee \msch P$ and the map $\msch f$ is, up to isomorphism, the embedding $\msch P\subset\msch  M'\vee \msch P$, which is $K$-admissible. Thus the lemma is valid for affine schemes. From this, the lemma is immediate for $\Mo$-schemes.

For an arbitrary $\Fun$-scheme $\cX$, assume that $\cM=(\msch M, M, \epsilon_M), \cN=(\msch N, N, \epsilon_N)$ and $\cP=(\msch P, P, \epsilon_P)$. We just showed that the push-out $\msch N\cup_{\msch M} \msch P$ exists and the map $\msch f$ is $K$-admissible in $\Bun \msch X$. Similar statements for $N\cup_{M}P$ and $f$ in $\Bun X$ are standard. We only need to show that these two pairs of data are compatible when compared on $\msch X_\Z.$ This is easily checked because both  base extension to $\Z$ functor and the pull back functor $e_X^*$ preserve colimits. 
\end{proof}

We define the \emph{$K$-theory space of $\cX$}, denoted as $\cK(\cX)$,  as $$\Omega|S_{\bullet} (\Bun \cX)|$$ where we use $\bullet$ to denote the simplicial degree.   
Using Lemma \ref{cobasechange} and the same proof as in \cite[section 1.9]{W85}, we obtain that $|S_{\bullet} (\Bun \cX)|$ is weakly homotopy equivalent to $|Q\Bun\cX|$ as a topological space. This shows that the $K$-theory of an $F_1$-scheme is in fact an infinite loop space and allows us to talk about the $K$-theory spectrum of an $\Fun$-scheme.

\subsubsection{The $K$-theory of $\Fun[H]$ and the Burnside ring}


 Let $\cX=\Spec(\Fun[H])$ where $H$ is an abelian group. By Corollary \ref{Cor_equiv_A-Mod_and_Coh_X} and  Proposition \ref{prop:ModulesOverMoSchemes},  the category  $\Coh\cX$ is equivalent to the category of finite $H$-sets  with $H$-equivariant maps. Since the admissible exact sequences in $\Coh\cX$ split, the proof in \cite[sections~1.8, 1.9]{W85} applies to show that the $K$-theory of $\cX$ is equivalent to the $K$-theory of the groupoid category of finite $H$-sets.   Carlsson, Douglas and Dundas show in \cite{CDD} that the $K$-theory of this category is equivalent to the $H$-fixed points of the equivariant sphere spectrum. In particular, we have
$$G_0(\Spec(\Fun[H]]) \cong \Omega[H]$$
where $\Omega[H]$ is the Burnside ring of the group $H$. It is hard to describe higher $G$-groups of $\Spec(\Fun[H])$, but we can use stable homotopy groups to characterize all K-groups of $\Spec(\Fun[H])$.

\begin{theorem}\label{KTheoryOfSpecF1[H]} Let $H$ be an abelian group and let $\Fun[H]$ be the associated monoid. Let $\Spec(\Fun[H])$ be the $\Fun$-scheme $(\Spec\Fun[H], \Spec \Fun[H]_\Z, \id)$, then we have $K_i(\Spec\Fun[H])\cong \pi_i^s(BH_+)$, where $\pi^s_i$ denotes the $i$-th stable homotopy group. In particular $K_i(\Spec(\Fun))\cong \pi^s_i(S^0)$.
\end{theorem}
\begin{proof} This theorem is based on facts that are well known to topologists while the detailed proofs are missing in the literature. We only give an outline of the complete proof in order to avoid digression into algebraic topology.

By Propositions  \ref{prop:DescriptionOfProjectiveAsets} and \ref{prop:ModulesOverMoSchemes}, we see that the category $\Bun(\Spec(\Fun[H]))$ is naturally equivalent to the category of based free $H$-sets of finite cardinality. Let $\mathcal{S}$ denote the associated groupoid of $\Bun(\Spec(\Fun[H]))$. That is, $\mathcal{S}$ has the same objects as $\Bun(\Spec(\Fun[H]))$, but has only isomorphisms in $\Bun(\Spec(\Fun[H]))$. Since exact sequences in $\Bun(\Spec(\Fun[H]))$ split, $K$-theory of $\Spec(\Fun[H])$ is the $K$-theory of the category of based free $H$-sets of finite cardinality. The general machinery of \cite{Segal} then applies to show that $K$-theory of $\Spec(\Fun[H])$ is the infinite loop space $Q(BH_+):=\Omega^\infty\Sigma^\infty BH_+$. In the special case that $H$ is the trivial group and $BH_+$ is the zero sphere  $S^0$, the fact that $K$-theory of finite sets  is weakly homotopy equivalent to $QS^0$ is known as the Barratt-Priddy-Quillen Theorem (see \cite{Barratt, Priddy}).

We summarize the previous discussions by the following chain of isomorphisms:

\begin{eqnarray*}  
  K_i(\Spec(\Fun[H])) \quad & \simeq & \quad \pi_{i+1}\Big(|Q\Bun\big(\Spec(\Fun[H])\big)|\Big)  \\ & \simeq & \quad  \pi_{i}\big(Q(BH_+)\big)  \\ & \simeq &\quad \pi^s_i(BH_+)
\end{eqnarray*}
 where $\pi^s_i(BH_+)$ denotes the $i$-th stable homotopy group of $BH_+$.
\end{proof}

\subsubsection{$K$-theory as a summand of $G$-theory}

 If $\cX=\Spec(\Fun)$, it follows from the definitions that $G_i(\cX)=K_i(\cX)$ because $\Bun\cX=\Coh\cX$.  This coincides with the fact that $\Fun$ is a ``field'' and  usual $G$-theory equals $K$-theory over a field. If  $\cX=\Spec(\Fun[H])$ where  $H \cong \Z/p$ for some prime number $p$ so that $H$ is simple, then $Q\Coh\cX$ is equivalent to the product category $Q\Coh\Spec(\Fun) \times Q\Bun\cX$. So $G_i(\cX)\cong K_i(\Fun)\oplus K_i(\cX)=\pi^s_i(S^0)\oplus \pi^s_i(BH_+)$ by Theorem \ref{KTheoryOfSpecF1[H]}. Note that this splitting is the tom Dieck splitting \cite{LMS, tom} of the equivariant sphere spectrum in the case when the group acting is simple.

  More generally, if $\cX=\Spec(\F_1[H])$ for any finite abelian group, then $\Bun\cX$ is the category of all free finite $H$-sets and $\Coh \cX$ is the category of all  finite $H$-sets. Since every finite $H$-set is isomorphic to copies of left cosets of $H$ for various subgroups,  one checks that the inclusion of the category $Q\Bun \cX$ into $Q\Coh\cX$ splits off, so the group  $K_*(\cX)$  splits off from $G_*(\cX)$ in this case.  But, if $H$ is not simple, it is hard to characterize the quotient group $G_*(\cX)/K_*(\cX)$.

 In general, $K$-theory does not embed into $G$-theory (cf.\ Theorem \ref{thm: image of K0 of Pn in G0}). However, there is hope that the image of $K$-theory in $G$-theory splits off (cf.\ Theorem \ref{thm: G0 of the projective line}).

\subsubsection{Comparison with $K$-theory of $\Z$-schemes}

 Let $\cX=(\msch X, X, e_X)$ be an $\Fun$-scheme. It follows from the constructions that we have natural maps $K_*(\cX)\longrightarrow K_*(\msch X)$, $K_*(\cX)\longrightarrow K_*(X)$  and $K_*(\cX)\longrightarrow K_*(\msch X_\Z).$ The map $K_*(\cX)\longrightarrow K_*(X)$ factors through $K_*(\cX)\longrightarrow K_*(\msch X_\Z)$. In the case that $\cX$ is given by an $\Mo$-scheme, the map $K_*(\cX)\longrightarrow K_*(\msch X_\Z)$ has been studied by many authors, for example see \cite{DFM}. The map $K_*(\Fun) \to K_*(\Z) \to K_*(\F_p)$ is homotopically equivalent to the projection onto the image of the $J$-homomorphism after localization at a prime $l\neq p$ (see \cite{Quillen551}).

\subsubsection{Examples}

\subsubsection*{The $K$-theory of the affine line}

 Let $\cX=\A^1_\Fun$ be the spectrum of $\Fun[T]$. Then $\Bun\cX$ is the category of free $\Fun[T]$-sets as $\Fun[T]$ does not have any non-trivial idempotents. Since the only isomorphism from $\Fun[T]$ to $ \Fun[T]$ is the identity, the set of isomorphism class of free $\Fun[T]$-sets  of rank $k$ is given by $\Sigma_k$. So we see that  $K(\Fun[T])\cong K(\F_1)$, if we use Segal's machinery \cite{Segal} to compute $K$-groups as explained in the proof of Theorem \ref{KTheoryOfSpecF1[H]}. This is consistent with the fact that   $K$-theory is homotopy invariant over regular $\Z$-schemes. \\
We calculate the $K_0$-term of projective spaces in Section \ref{subsection: G0 and K0 of projective space}.

\subsubsection*{The $K$-theory of a monoid with a non-trivial idempotent}

 As another example, consider $\cX=\Spec A $ for $A=\{0,1,e\}$ where $e^2=e$ is an idempotent. The monoid $A$ has thus the two non-zero idempotents $e$ and $1$. Any finite projective $A$-set is isomorphic to $(\vee_{i\in I} eA) \vee (\vee_{j\in J} A)$, for some finite indexing sets $I$ and $J$. The set of isomorphisms of such a set is equivalent to  $\Sigma_n \times \Sigma_m$ where $n$, $m$ are cardinalities of $I$ and $J$. This is because there is no isomorphism from $eA$ to $A$. Using Segal's machinery \cite{Segal}, we see that the $K$-theory spectrum of $\cX$ is equivalent to the product of the sphere spectrum with itself, that is, $K(\cX)= Q\S^0 \times Q\S^0$. 


\subsection{$K$-theory as a ring spectrum}

As proved in the following Proposition \ref{biexact},  the category $\Bun \cX$ is strict symmetric monoidal with a tensor product for any $\Fun$-scheme $\cX$. This yields a pairing on the $K$-theory of an $\Fun$-scheme (cf.\ \cite[section 1.5]{W85}), i.e.\ a graded commutative ring structure on $K_*(\cX)=\pi_*(\Omega|S_{\bullet} (\Bun \cX)|)$. More precisely, we show in this section that $K$-theory is a contravariant functor from the category of $\Fun$-schemes into the category of symmetric ring spectra. A comparison with $K$-theory for bipermutative categories \cite{EM} yields that the  K-theory ring spectrum  is equivalent to the  sphere spectrum in the case of $\cX=\Spec\Fun$. We end this paper with a calculation of $K_0$ and $G_0$ of projective space.

 We shall explain that the $K$-theory spectrum  is a ring spectrum. Since $K$-theory naturally takes values in the category of symmetric spectra (cf.\ \cite{GH}), we choose this category to be our model of spectra. We refer to \cite{HoShSm} for definitions and further details on symmetric spectra.

\begin{proposition}\label{biexact}
 Let $\cX=(\msch X, X, e_X)$ be an $\Fun$-scheme. Then there is a  tensor product  $$\otimes:\quad \QCoh\cX \times \QCoh\cX \longrightarrow \QCoh\cX$$ which sends the pair $(\cM=(\msch M, M, \epsilon_M), \ \cN=(\msch N, N, \epsilon_N))$ to $$\cM\otimes\cN=( \msch M \otimes \msch N, M \otimes N, \epsilon_M \otimes \epsilon_N).$$ The category $\Bun\cX$ is closed under the tensor product. Moreover,  the tensor product is bi-exact on $\Bun\cX$, which means:
\begin{enumerate}\item[(1)]
$\0\otimes \cM =\0$ for any $\cM\in \Bun\cX$;
\item[(2)] for any $\cM\in \Bun\cX$, $\cM\otimes-$ and $-\otimes \cM$ send isomorphisms to isomorphisms, $K$-admissible monomorphisms to $K$-admissible monomorphisms and push-outs along $K$-admissible monomorphisms to push-outs along $K$-admissible monomorphisms;
\item[(3)] for any $K$-admissible monomorphisms $f:\cM \to \cM'$ and $g: \cN \to \cN'$, the natural map $f\otimes g: (\cM'\otimes \cN)\underset{(\cM\otimes \cN)}{\cup}(\cM\otimes \cN')\to (\cM'\otimes \cN')$ is a $K$-admissible monomorphism.
\end{enumerate}
\end{proposition}
\begin{proof} The functor $\otimes$ is well defined because $(\msch M \otimes \msch N)_\Z = \msch M_\Z \otimes \msch N_\Z$  by Lemma \ref{base_extension_properties_for_sheaves}, which is isomorphic to $e_X^*M \otimes e_X^*N$ (through $\epsilon_M \otimes \epsilon_N$). This is further isomorphic to $e_X^*(M \otimes N)$. All the statements of the proposition can be proven in a straightforward manner by considering them separately for $X$ and $\msch X$.
\end{proof}

\subsubsection{Functoriality of $K$-theory}

Let $\Phi=(\msch\varphi,\varphi): \cX=(\msch X, X, e_X) \to \cY=(\msch Y,Y,e_Y)$ be a morphism of  $\Fun$-schemes. We have  the following commutative diagram of $\Z$-schemes
 $$\xymatrix{\msch X_\Z\ar[rr]^{\msch\varphi_\Z}\ar[d]_{e_X}&& \msch Y_\Z\ar[d]^{e_Y}\\ X\ar[rr]^{\varphi} && Y.} $$
 Let $\cM=(\msch M, M,\epsilon_{M})$ be an $\cO_\cY$-module. We have the following isomorphism of $\cO_{\msch Y_\Z}$-modules
 $$\xymatrix{ \msch M_\Z    \ar[rr]^{\epsilon_M} && e_Y^*M.  }$$ Pulling back along $\msch \varphi _\Z$, we have an isomorphism of $\cO_{\msch X_\Z}$-modules
 $$\xymatrix{ \msch \varphi_\Z ^* \msch M_\Z    \ar[rr]^{\msch \varphi_\Z ^*\epsilon_M} &&\msch \varphi_\Z ^*e_Y^* M. }$$
One checks easily that $\msch \varphi_\Z ^* \msch M_\Z = (\msch \varphi ^* \msch M)_\Z$. One also has  $\msch \varphi_\Z ^*e_Y^* M =  e_X^*\varphi^*M $ because $e_Y\msch\varphi_\Z = \varphi e_X$. So the isomorphism ${\msch \varphi_\Z ^*\epsilon_M}$ can be rewritten as 
$$\xymatrix{ (\msch \varphi ^* \msch M)_\Z    \ar[rr]^{\msch \varphi_\Z ^*\epsilon_M} &&e_X^*(\varphi^*M). }$$
This shows that $(\msch \varphi ^* \msch M, \varphi^*M, {\msch \varphi_\Z ^*\epsilon_M}  )$ is an $\cO_{\cX}$-module, which is denoted as $\Phi^*\cM$ and called the pull back  of $\cM$. Pull back of a morphism $(\msch f, f) $ along $\Phi$ is defined componentwise as $(\msch\varphi^*\msch f, \varphi^*f)$.  One checks that $\Phi^*$ is a functor, which sends quasi-coherent sheaves to quasi-coherent sheaves and sends locally free sheaves to locally free sheaves. 
\begin{theorem}\label{ring} Let $\cX=(\msch X, X, e_X)$ be an $\Fun$-scheme. Then the $K$-theory of $\cX$ is a symmetric ring  spectrum, denoted as  $\cK(\cX)$. A morphism between $\Fun$-schemes  $\Phi: \cX \to \cY$ induces a natural morphism of symmetric ring spectra $\Phi^*: \cK(\cY) \to \cK(\cX)$. 
\end{theorem}

\begin{proof} Proposition \ref{biexact} shows that the category $\Bun\cX$ is bi-exact for any $\Fun$-scheme $\cX$. We showed already that we can apply Waldhausen's $S_{\bullet}$-construction in this setting. Analogous to \cite[Proposition 6.11]{GH}, it follows that $K$-theory is a symmetric ring spectrum. A morphism between $\Fun$-schemes  $\Phi: \cX \to \cY$ induces a functor $\Phi^\ast :\Bun \cY \to \Bun \cX$. It is easy to check that $\Phi^*$ respects the tensor product structure on $\Bun \cX$ and $\Bun \cY$, and that it sends cofibrations, i.e., admissible monomorphisms, to cofibrations. Further note that the $S_{\bullet}$-construction is functorial. The verification that the map on $K$-theory of symmetric ring spectra is a map of ring spectra follows from the construction of the ring structure on the $K$-theory spectrum \cite[Appendix 6]{GH}. 
\end{proof}

\subsubsection{K-theory of bipermutative categories and the $S_\bullet$-construction}

 By Proposition \ref{prop:SplittingOfAdmissibleShortExactSequences}, admissible short exact sequences in $\Bun \cX$ split if $\cX$ is the spectrum of a monoid. This fact generalizes to all $\cM_0$-schemes.

\begin{theorem}\label{thm: Bun X is split}
 Let $\cX$ be an $\cM_0$-scheme. Then admissible short exact sequences in the category $\Bun\cX$  split.
\end{theorem}

\begin{proof}
 The proof is based on the following fact. Let $A$ be a monoid and $M$, $N$ and $P$ be projective $A$-sets fitting into an admissible short exact sequence
 \[
  0 \quad \longrightarrow  \quad M  \quad {\longrightarrow} \quad  N  \quad {\longrightarrow} \quad  P \quad  \longrightarrow  \quad 0.
 \]
By Proposition \ref{prop:SplittingOfAdmissibleShortExactSequences}, this sequence is isomorphic to  the canonical short exact sequence
 \[
  0 \quad \longrightarrow  \quad M  \quad {\longrightarrow} \quad  M\vee P  \quad {\longrightarrow} \quad  P \quad  \longrightarrow  \quad 0.
 \]
 Obviously, it admits a unique section $\sigma: P\to M\vee P$.

 Let $\cX$ be an $\cM_0$-scheme and $0\ \to\ \cM \ \to \ \cN \ \to \ \cP \ \to \ 0$ a short exact sequence of locally projective sheaves on $\cX$. Then this sequence splits locally. Since all the locally defined sections are unique, they glue to a section $\tilde\sigma:\cP\to\cN$.
\end{proof}

Elmendorf and Mandell \cite[Corollary 3.9]{EM} show that the $K$-theory of bipermutative small categories is an $\text{E}_\infty$-symmetric ring spectrum, that is, it is equivalent to a commutative symmetric ring spectrum. Note that this defines an $E_\infty$-symmetric ring spectrum structure on $K(\Fun)$ and  this is equivalent to the sphere spectrum, which is a commutative symmetric ring spectrum. It is easy to check that $\Bun \cX$ is bipermutative if $\cX$ is  the spectrum of a monoid.  When the admissible short exact sequences in the category   $\Bun \cX$ split,  the $K$-theory ring spectrum obtained by the Waldhausen's $S_{\bullet}$-construction and the one using bipermutative categories are equivalent, which is essentially proved in \cite[Theorem 4.3]{Weibel}.  As an immediate consequence, we have the following stronger version of Theorem \ref{KTheoryOfSpecF1[H]}. 

\begin{corollary}\label{KTheoryOfSpecF1[H]AsRingSpectra} Let $H$ be an abelian group. The $K$-theory spectrum  of $\Spec(\Fun[H])$ is an  $\text{E}_\infty$-symmetric ring spectrum and hence equivalent to a commutative symmetric ring spectrum. In particular, the $K$-theory ring spectrum of $\Spec(\Fun)$ is equivalent to the sphere spectrum as a commutative ring spectrum. \qed
\end{corollary}

\subsubsection{On $G_0$ and $K_0$ of projective space}
\label{subsection: G0 and K0 of projective space}

While an explicit calculation of all terms of the $G$-theory and $K$-theory of an $\Fun$-scheme is extremely difficult---even the most simple case of $X=\Spec\Fun$ leads to the stable homotopy groups of the sphere---the zeroth terms are accessible via the Grothendieck groups generated by coherent sheaves resp.\ locally projective sheaves modulo the relations defined by admissible sequences. 

\begin{theorem}\label{thm: K0(X) is generated by Pic X}
 Let $X$ be an integral $\cM_0$-scheme and $\Pic X$ its Picard group. Then $K_0(X)$ is naturally isomorphic to the group ring $\Z[\Pic X]$. In particular, $K_0(X)$ is freely generated by the set of isomorphism classes of line bundles as an abelian group.
\end{theorem}

\begin{proof}
 First note that all coordinate monoids of $\cX$ are integral and thus without non-trivial idempotents. This means that every locally projective sheaf $\cM$ over $\cX$ is locally free. 

 We show by induction on the rank $n$ of $\cM$ that $\cM$ decomposes uniquely as a sum of line bundles. If $n=1$, there is nothing to prove. If $n>1$, then $\cM$ has a subline bundle $\cL$ for the following reason. Since $X$ is integral, it has a unique generic point $\eta$. The choice of a  rank one free submodule in the stalk $\cM_\eta$ defines a subline bundle $\cL$ of $\cM$ by defining $\cL(U)$ as the intersection of $\cM(U)$ with $\cL_\eta$ inside $\cM_\eta$ where $U$ ranges through all open subsets of $X$. 

 Clearly, the quotient $\cM/\cL$ is locally projective. This means that we obtain an admissible sequence
 \[
  0 \quad \longrightarrow  \quad \cL  \quad {\longrightarrow} \quad  \cM  \quad {\longrightarrow} \quad  \cM/\cL \quad  \longrightarrow  \quad 0
 \]
 in $\Bun X$. By Theorem \ref{thm: Bun X is split}, this sequence splits, i.e.\ $\cM\simeq\cL\vee\cN$ for some locally projective sheaf $\cN$, which is of rank $n-1$. By the induction hypothesis, $\cN$ decomposes into line bundles, which proves the latter claim of the theorem.

 The ring structure of $K_0(X)$ is induced by the tensor product of locally projective sheaves. Since $\Pic X$ freely generates $K_0(X)$ as an abelian group, it is clear that $K_0(X)$ is nothing else than the group ring $\Z[\Pic X]$.
\end{proof}

Let $\cO(n)=\cO_X(n)$ be the twisted sheaf for the projective space $\P^n$. We remark that they are defined over $\Fun$ since a line bundle over $\P^n_\Z$ trivializes over all open sets of the canonical atlas and all their intersection, because all these open subsets are of the form $\A^r_\Z\times \G_{m,\Z}^{n-r}$, a scheme that has only trivial line bundles. Therefore every line bundle over $\P^n$ is toric and descends to $\Fun$ by Deitmar's theorem \cite[Thm.\ 4.1]{Deitmar07}. Also cf.\ Section \ref{coherent_examples} where we examine in detail how the gluing data descends to $\Fun$ in case of the projective line $\P^1$. As an immediate consequence of the previous theorem we derive the following statement.

\begin{corollary}
 The ring $K_0(\P^n_\Fun)$ is isomorphic to $\Z[\cO(n)]_{n\in Z}$ where the product is defined by the rule $\cO(n)\otimes\cO(m)\simeq\cO(n+m)$. \qed
\end{corollary}

The $G_0$-term is harder to compute. In case of $\P^1$ one can use the description of all finitely generated $A$-sets for $A=\Fun[T]$ (cf.\ Section \ref{A-set_examples}) to derive the following characterization. A detailed proof can be found in \cite[Thm.\ 4]{S10}.

\begin{theorem} \label{thm: G0 of the projective line}
 The abelian group $G_0(\P^1_\Fun)$ is freely generated by the class of $\cO$, the class of $\cO(1)$ and countably infinitely many classes $\cC_n$ of coherent sheaves that are not locally projective. \qed
\end{theorem}

This behaviour contrasts the classical result 
\[
 G_0(\P^{n}_\Z) \ = \ K_0(\P^{n}_\Z) \ = \ K_0(\A^n)^{n+1} \ \simeq\Z^{n+1},
\]
for the $G$-theory and $K$-theory of projective spaces. The reason is easily explained: the short exact sequence
\[
  0 \quad \longrightarrow  \quad \cO_\Z  \quad {\longrightarrow} \quad  \cO(1)_\Z\oplus\cO(1)_\Z  \quad {\longrightarrow} \quad  \cO(2)_\Z \quad  \longrightarrow  \quad 0
\]
is not defined over $\Fun$ since it does not split. However, we see by the above theorem that the image of $K_0(\P^1_\Fun)$ in $G_0(\P^1_\Fun)$ is of rank $2$ and generated by $\cO$ and $\cO(1)$. This generalizes to all projective spaces.

\begin{theorem} \label{thm: image of K0 of Pn in G0}
 The image $\tilde K_0(\P^n_\Fun)$ of the canonical map $K_0(\P^{n}_\Fun)\to G_0(\P^n_\Fun)$ is freely generated by $\cO, \dotsc, \cO(n)$ as an abelian group, and the base extension $G_0(\P^n_\Fun)\to G_0(\P^{n}_\Z)$ restricts to an isomorphism
 \[
  \tilde K_0(\P^n_\Fun) \ \stackrel{\sim}{\longrightarrow} \ G_0(\P^{n}_\Z) \ \simeq \ K_0(\P^{n}_\Z).
 \]
\end{theorem}

\begin{proof}
 The result is trivial for $\P^0_\Fun$. Thus we may assume that it holds for $\P^{n-1}_\Fun$ and prove it by induction for $\P^n_\Fun$. Consider a closed subscheme $Z\simeq\P^{n-1}_\Fun$ of $\P^n_\Fun$ and let $\cK(m)$ be the torsion sheaf with support $Z$ that is the extension of the twisted sheaf $\cO_Z(m)$ of $Z$ by $0$. Then we have for every $m\in\Z$ an exact series
 \[
  0 \quad \longrightarrow  \quad \cO(m-1)  \quad {\longrightarrow} \quad  \cO(m)  \quad {\longrightarrow} \quad  \cK(m) \quad  \longrightarrow  \quad 0
 \]
 in $\Coh \P^n_\Fun$. This yields the relations $\cK(m)=\cO(m)-\cO(m-1)$ in $G_0(\P^n_\Fun)$. We show by a nested induction on $l\geq n+1$ that every sheaf $\cO(l)$ can be expressed by a linear combination of $\cO,\dots,\cO(n)$. Namely, $\cO(l)=\cO(l-1)+\cK(l)$ where the sheaf $\cO(l-1)$ can be expressed as a linear combination of $\cO,\dots,\cO(n)$ by induction hypothesis on $l$ and where $\cK(l)$ can be expressed as a linear combination of $\cK(0),\dots,\cK(n-1)$ by induction hypothesis on $n$, which in turn equals a linear combination of $\cO,\dots,\cO(n)$. A similar argument shows that all $\cO(l)$ with $l\leq -1$ can be written as a linear combination of $\cO,\dots,\cO(n)$.

 This shows that $\cO,\dotsc,\cO(n)$ generate $\tilde K_0(\P^n_\Fun)$. The rest of the theorem follows from the fact that $G_0(\P^n_\Z)$ is freely generated by $\cO_\Z,\dotsc,\cO(n)_\Z$.
\end{proof}

\bibliographystyle{plain}

\end{document}